\documentclass[11pt]{article}

\usepackage{graphicx}
\usepackage{epsfig}
\usepackage{amsmath, amsthm, amssymb, mathtools}
\usepackage{enumitem}
\usepackage[mathscr]{euscript}
\usepackage[english]{babel}
\usepackage{aliascnt}
\usepackage[all,cmtip]{xy}
\usepackage{url}
\usepackage{tikz-cd}
\usepackage{hyperref}
\usepackage[capitalize]{cleveref}

\def\ds{\displaystyle}

\renewcommand\S[1]{\mathrm{S}_{#1}}
\newcommand\A[1]{\mathrm{A}_{#1}}

\def\Hom{\mathrm{Hom}}

\def\Bij{\mathrm{Bij}}

\def\Or{\mathrm{Or}}

\def\det{\mathrm{det}}
\def\norm{\mathrm{Nm}}
\def\sgn{\mathrm{sgn}}
\def\trace{\mathrm{Tr}}
\def\qdet{\mathrm{qdet}}
\def\id{\mathrm{id}}

\def\Q{\mathbb{Q}}

\def\Z{\mathbb{Z}}
\def\Pf{\mathrm{Pf}}

\newcommand{\fundgroup}[1]{\pi(#1)}

\newcommand{\B}[1]{\operatorname{\mathbf{B}{#1}}}
\newcommand{\F}[1]{\ensuremath{\mathbb{F}_{#1}}}
\newcommand{\extpower}{{\textstyle\bigwedge}}
\renewcommand{\subset}{\subseteq}

\newcommand{\ferrand}{\Phi}
\def\eps{\varepsilon}

\def\simto{\ds\mathop{\longrightarrow}^\sim\,}

\def\into{\hookrightarrow}

\newcommand{\Rbase}{{R_1}}
\newcommand{\Rbasetwo}{{R_2}}
\newcommand{\Xbase}{{X_1}}
\newcommand{\Xbasetwo}{{X_2}}
\newcommand{\transpose}[1]{#1^{\top}}
\newcommand{\set}[1]{\{1,\dots,#1\}}

\newcommand{\fixpower}[4][]{({#2}^{\otimes_{#1} #3})^{#4}}

\newcommand{\loosdiscquad}[2][]{\mathfrak{D}_{#1}(#2)}
\newcommand{\loosdiscalg}{\mathrm{Dis}}
\newcommand{\conjugate}[2]{#1^{(#2)}}
\newcommand{\permgroup}[1]{\mathrm{S}_{#1}}
\newcommand{\altgroup}[1]{\mathrm{A}_{#1}}

\newcommand{\floor}[1]{\left\lfloor#1\right\rfloor}
\newcommand{\disc}{\delta}
\newcommand{\discalg}{\Delta}
\newcommand{\stdi}[1]{\tau}

\newcommand{\basis}{\mathrm{e}}

\newcommand{\cat}[1]{\operatorname{\mathbf{#1}}}
\newcommand{\mb}[1]{\mathbb{#1}}

\newcommand{\Aff}{\cat{Aff}}

\newcommand{\Alg}[1]{\cat{Alg}_#1}
\newcommand{\Et}[1]{\cat{\acute{E}t}_#1}

\def\midotimes{\bigotimes}

\newtheorem{theorem}{Theorem}[section] 
\crefname{theorem}{Theorem}{Theorems}

\newtheorem{lemma}[theorem]{Lemma}
\crefname{lemma}{Lemma}{Lemmas}

\crefname{proposition}{Proposition}{Propositions}

\newtheorem{corollary}[theorem]{Corollary}
\crefname{corollary}{Corollary}{Corollaries}

\theoremstyle{definition} 

\newtheorem{definition}[theorem]{Definition}
\crefname{definition}{Definition}{Definitions}

\newtheorem{remark}[theorem]{Remark}
\crefname{remark}{Remark}{Remarks}

\newtheorem{example}[theorem]{Example}
\crefname{example}{Example}{Examples}

\author{Owen Biesel and Alberto Gioia}
\begin{document}
\title{Isomorphisms of Discriminant Algebras}
\maketitle
\begin{abstract}
 For each natural number $n$, we define a category whose objects are \emph{discriminant algebras in rank $n$}, i.e.\ functorial means of attaching to each rank-$n$ algebra a quadratic algebra with the same discriminant. 
 We show that the discriminant algebras defined in \cite{Bie15}, \cite{LoosDiscAlg}, and \cite{Rost} are all isomorphic in this category, and prove furthermore that in ranks $n\leq 3$ discriminant algebras are unique up to unique isomorphism.
\end{abstract}
\tableofcontents

\section{Introduction and definitions}

A \emph{discriminant algebra} in rank $n$ is an isomorphism- and base-change-preserving way of associating to each ring $R$ and rank-$n$ algebra\footnote{i.e.\ projective of constant rank $n$ as an $R$-module} $A$  a rank-$2$ algebra with the same discriminant bilinear form as $A$.
Discriminant algebras have been constructed in rank $3$ by Markus Rost \cite{Rost}, in arbitrary rank by Ottmar Loos in \cite{LoosDiscAlg}, and in rank $\geq 2$ by the authors in \cite{Bie15}.
The goal of this paper is to show that these three constructions are in some sense isomorphic; in this section we define the category of discriminant algebras in order to make the notion of isomoprhism precise.

\begin{definition}
Form the following categories, $\Aff$, $\cat{Disc}$, and $\Alg{n}$ for each natural number $n$, as follows:
\begin{itemize}
\item $\Aff$ is the category of affine schemes, the opposite of the category of commutative rings.
\item $\cat{Disc}$ is the category of discriminant data.
It has as objects triples $(R, L, d)$, where $R$ is a commutative ring, $L$ is a locally free rank-$1$ $R$-module, and $d\colon L^{\otimes 2}=L\otimes_R L\to R$ is an $R$-module homomorphism.
A morphism from $(R',L',d')$ to $(R,L,d)$ consists of a ring homomorphism $R\to R'$ and an $R'$-module isomorphism $R'\otimes_R L\simto L'$ such that the following square commutes:
\[\begin{tikzcd}R'\otimes_R (L^{\otimes 2}) \arrow{d}{\id_{R'}\otimes d}\arrow{r}{\sim} & L'^{\otimes_{R'} 2} \arrow{d}{d'}\\
R'\otimes_R R \arrow{r}{\sim} & R'.\end{tikzcd}\]
The forgetful functor $\cat{Disc}\to\Aff : (R,L,d)\mapsto R$ makes $\cat{Disc}$ a category fibered in groupoids (CFG) over $\Aff$.
\item $\Alg{n}$ is the category of rank-$n$ algebras.
Its objects are pairs $(R,A)$ with $A$ an $R$-algebra that is locally free of rank $n$ as an $R$-module.
A morphism $(R',A')\to(R,A)$ consists of a ring homomorphism $R\to R'$ and an $R'$-algebra isomorphism $R'\otimes_R A\simto A'$.
The forgetful functor $\Alg{n}\to\Aff : (R,A)\mapsto R$ makes $\Alg{n}$ a CFG over $\Aff$.
\end{itemize} 
\end{definition}
There is a map of CFGs $\extpower^n\colon\Alg{n}\to\cat{Disc}$ over $\Aff$ sending a rank-$n$ algebra to its discriminant data: $(R,A)\mapsto (R, \extpower^n A, \disc_{A})$, where
\begin{align*}
 \disc_A &\colon \extpower^n A\otimes \extpower^n A \to R\\
 &\colon (a_1\wedge\dots\wedge a_n) \otimes (b_1\wedge\dots\wedge b_n)\mapsto \det\bigl(\trace(a_ib_j)\bigr)_{ij}
\end{align*}
is the \emph{discriminant bilinear form} of $A$.

\begin{definition}\label{discalg}
 A \emph{discriminant algebra} in rank $n$ is a factorization of $\extpower^n\colon\Alg{n}\to\cat{Disc}$ via $\Alg{2}$.
 More specifically, it consists of a morphism $D$ of CFGs $\Alg{n}\to\Alg{2}$ over $\Aff$ together with a $2$-isomorphism filling the following triangle of morphisms of CFGs:
 \[\xymatrix{
  \Alg{n} \ar[rr]^{\extpower^n} \ar[dr]_{D} & \ar@{<=}[d]^{\sim}& \cat{Disc}\\
  & \Alg{2}\ar[ru]_{\extpower^2} & .
}\]
 In other words, to exhibit a discriminant algebra we must associate to every rank-$n$ $R$-algebra $A$ a rank-$2$ $R$-algebra $D(R,A)$ together with functorial base-change isomorphisms $R'\otimes_R D(R,A)\cong D(R',R'\otimes_R A)$, and we must exhibit a natural isomorphism $\extpower^2 D(R,A)\cong \extpower^n A$ that identifies the discriminant forms of $A$ and $D(R,A)$.

The \emph{category of discriminant algebras} is then just the category of such factorizations of $\extpower^n$ via $\extpower^2$. In particular, an \emph{isomorphism} of discriminant algebras is a natural isomorphism of functors $\Alg{n}\to\Alg{2}$ over $\Aff$ making the resulting triangle of natural isomorphisms commute.
In other words, an isomorphism between two discriminant algebras $D$ and $D'$ is an isomorphism $D(R,A)\cong D'(R,A)$ that respects base change, and that also induces the composite isomorphism $\extpower^2 D(R,A) \cong \extpower^n A \cong \extpower^2 D'(R,A)$ coming from the data of the two discriminant algebras. 
\end{definition}

\section{Properties of discriminant algebras}

\subsection{Existence}

The first important fact about discriminant algebras is that they exist.
We review the construction of one such morphism from \cite{Bie15}.
Let $R$ be a ring and $A$ a rank-$n$ $R$-algebra; the \emph{Ferrand homomorphism} $\ferrand_{A/R}\colon\fixpower[R]{A}{n}{\S{n}}\to R$ is the unique $R$-algebra homorphism such that for each $R$-algebra $R'$ and element $a\in A'\coloneqq  R'\otimes_R A$, the composite homomorphism
\[\fixpower[R']{A'}{n}{\S{n}} \cong R'\otimes_R \fixpower[R]{A}{n}{\S{n}}\to R'\otimes_R R \cong R'\]
sends $a\otimes\dots\otimes a$ to $\norm_{A'/R'}(a)$.
(This is the algebra homomorphism representing the multiplicative homogeneous degree-$n$ polynomial law $\norm\colon A\to R$ in the sense of \cite{MultHomMaps}.)
Then define
\[\discalg(R,A) \coloneqq  \fixpower{A}{n}{\A{n}}\ \midotimes_{\mathclap{\fixpower{A}{n}{\S{n}}}}\ R,\]
where the maps defining the tensor product are the inclusion $\fixpower{A}{n}{\S{n}}\into\fixpower{A}{n}{\A{n}}$ and the Ferrand homomorphism $\fixpower{A}{n}{\S{n}}\to R$.
By \cite[Theorem 4.1]{Bie15}, this is a quadratic $R$-algebra, and we have the following theorem:

\begin{theorem}\label{existence}
 Let $n\geq 2$ be a natural number.
 For each ring $R$ and rank-$n$ $R$-algebra $A$, let $\Delta(R,A)$ be the rank-$2$ $R$-algebra defined above.
 This assignment constitutes a discriminant algebra $\Delta:\Alg{n}\to\Alg{2}$.
\end{theorem}

\begin{proof}
 Theorem 4.1 of \cite{Bie15} tells us that $\Delta(R,A)$ has rank $2$ and is equipped with a discriminant-preserving isomorphism between its top exterior power and that of $A$.
 Theorem 6.1 there gives us a canonical isomorphism $R'\otimes_R \Delta(R,A) \cong \Delta(R', R'\otimes_R A)$ for each $R$-algebra $R'$.
 Since this isomorphism is induced by the canonical isomorphism $R'\otimes_R (A^{\otimes_R n}) \cong (R'\otimes_R A)^{\otimes_{R'} n}$, it is easy to check that it is functorial with respect to multiple changes of base, and preserves the discriminant-preserving isomorphism between its top exterior power and $A$'s.
\end{proof}

Given any $R$-algebra $A$, it will be helpful to refer to the following elements of $A^{\otimes n}$ and $\fixpower{A}{n}{\S{n}}$:
\begin{itemize}
 \item For each $a\in A$ and $i\in\set{n}$, the element $\conjugate{a}{i}\in A^{\otimes n}$ is the pure tensor $1\otimes\dots\otimes 1\otimes a\otimes 1\otimes \dots\otimes 1$, with the $a$ in the $i$th position.
 \item For each $a\in A$ and $k\in\{0,\dots,n\}$, the element $e_k(a)\in \fixpower{A}{n}{\S{n}}$ is the sum
 \[e_k(a) := \sum_{\substack{i_1<\dots<i_k\\\in\set{n}}}\prod_{j=1}^k \conjugate{a}{i_j}.\]
 \end{itemize}
If $A$ is a rank-$n$ $R$-algebra so that we have the Ferrand homomorphism $\ferrand_{A/R}\colon \fixpower{A}{n}{\S{n}}\to R$, then the image of $e_k(a)$ in $R$ is the coefficient of $(-1)^k\lambda^{n-k}$ in $\norm_{A[\lambda]/R[\lambda]}(\lambda-a)$, the characteristic polynomial of $a$.
 This coefficient is denoted $s_k(a)$.
 In particular, $s_n(a)=\norm_{A/R}(a)$ and $s_1(a)=\trace_{A/R}(a)$, and $s_2\colon A\to R$ is called the \emph{quadratic trace} and is used in Loos's construction of a discriminant algebra.
 See \cite[Section 3]{Bie15} for more on $\ferrand_{A/R}$.

It will also be helpful to note that $\discalg(R,A)$ is generated as an $R$-module by $\{1\}\cup\{\gamma(a_1,\dots,a_n):a_1,\dots,a_n\in A\}$, where $\gamma(a_1,\dots,a_n)$ is the image under the map $\fixpower{A}{n}{\A{n}}\to \discalg(A,R)$ of the element
\[\sum_{\sigma\in\A{n}}\prod_{i=1}^n \conjugate{a_i}{\sigma(i)}.\]
If two of the $a_i$ are equal, or if we add $\gamma(a_1,a_2,\dots,a_n)$ to $\gamma(a_2,a_1,\dots,a_n)$, then the result is the image of an $\S{n}$-invariant element of $\fixpower{A}{n}{\A{n}}$ and is set equal to its image under $\ferrand_{A/R}\colon \fixpower{A}{n}{\S{n}}\to R$.
The discriminant-identifying isomorphism $\extpower^2 \discalg(R,A)\to \extpower^n A$ sends elements of the form $1\wedge\gamma(a_1,\dots,a_n)$ to $a_1\wedge\dots\wedge a_n$.

Note also that $\gamma(a_1,a_2,\dots,a_n)-\gamma(a_2,a_1,\dots,a_n)$ is the image in $\discalg(R,A)$ of the $\A{n}$-invariant tensor
\[\sum_{\sigma\in\S{n}}\sgn(\sigma)\prod_{i=1}^n\conjugate{a_i}{\sigma(i)}\in\fixpower{A}{n}{\A{n}}.\]

\subsection{Description when $2$ is a unit}

We also know the following about every discriminant algebra operation in the case that $2$ is invertible:

\begin{theorem}
 Let $D\colon \Alg{n}\to\Alg{2}$ be a discriminant algebra, and let $R$ be a ring in which $2$ is a unit.
 Then for each rank-$n$ algebra $A$, there is a canonical isomorphism $D(R,A)\cong R\oplus \extpower^n A$, where the multiplication on the latter module has identity $(1,0)$ and the product of $\xi,\psi\in \extpower^n A$ is $\disc_{A}(\xi,\psi)/4\in R$.
\end{theorem}
\begin{proof}
 Since $D(R,A)$ is a quadratic algebra over a ring in which $2$ is a unit, by \cite[Proposition 6.2]{Bie15} we have a canonical isomorphism $D(R,A)\cong R\oplus \extpower^2 D(R,A)$ with unit $(1,0)$ and  multiplication $\xi\cdot\psi = \disc_D(\xi,\psi)/4$.
 Using the discriminant-identifying isomorphism $\extpower^2 D\cong\extpower^n A$, we obtain the desired description of $D(R,A)$.
\end{proof}

\subsection{Description in the \'etale case}

We also know what a discriminant algebra does to a connected ring $R$ and finite \'etale algebra $A$.

\begin{theorem}
 Let $D\colon \Alg{n}\to\Alg{2}$ be a discriminant algebra with $n\geq 2$.
 Let $R$ be a ring equipped with a geometric point $R\to K$ and corresponding fundamental group $\fundgroup{R,K}$, and let $A$ be a rank-$n$ \'etale $R$-algebra, with corresponding $\fundgroup{R,K}$-set $X$.
 Then $D(R,A)$ is also \'etale, and corresponds to the $2$-element $\fundgroup{R,K}$-set $\Or(X)\coloneqq  \Bij(\set{n},X)/\A{n}$.
\end{theorem}

\begin{proof}
 A rank-$n$ $R$-algebra $A$ is \'etale if and only if its discriminant form $\disc_A\colon\extpower^n A\times\extpower^n A\to R$ is nondegenerate in the sense of inducing an isomorphism $\extpower^n A\to \Hom(\extpower^n A,R)$.
 Since $A$ and $D(R,A)$ have isomorphic discriminant forms, if the former is \'etale then so is the latter.
 Thus $D$ restricts to a functor $\Et{n}\to\Et{2}$, where $\Et{n}$ is the full subcategory of $\Alg{n}$ consisting of pairs $(R,A)$ with $A$ \'etale over $R$.
 
 Now $\Et{n}$ is equivalent to the fibered category $\B{\S{n}}$ of $\S{n}$-torsors, and we claim that there are only two morphisms of fibered categories $\B{\S{n}}\to \B{\S{2}}$ up to $2$-isomorphism.
 Indeed, such a morphism is determined up to $2$-isomorphism by an object of the fiber category $(\B{\S{2}})_\Z$ equipped with an action by $\S{n}$; the only such object is $\Z^2$, and the only such actions are the trivial action and the one coming from the sign homomorphism $\S{n}\to\S{2}$ composed with the action of $\S{2}$ on the factors of $\Z^2$.
 The trivial action corresponds to the assignment $(R,A)\mapsto (R, R^2)$, which does not admit a discriminant-identifying isomorphism of top exterior powers.
 Thus every discriminant algebra operation restricts to the same homomorphism $\Et{n}\to\Et{2}$, which by \cite[Theorem 4]{Bie15} is the one sending a rank-$n$ algebra corresponding to $X$ to the rank-$2$ algebra corresponding to $\Or(X)$.
\end{proof}

\section{Uniqueness in low rank}

In this section we will prove the following uniqueness theorem:

\begin{theorem}\label{uniqueness}
 For ranks $n\leq 3$, there is a unique discriminant algebra up to unique isomorphism.
\end{theorem}

We already have existence from \cref{existence}, so in the remainder of this section we will show that if $D$ and $D'$ are two discriminant algebras in rank $n\leq 3$, there is a unique isomorphism $D\cong D'$.

\subsection{Proof for $n\leq 1$}
In ranks $n=0$ and $n=1$, for each ring $R$ and rank-$n$ $R$-algebra $A$ there is a unique $R$-algebra isomorphism $A\cong R^n$, i.e.\ a unique morphism $(R,A)\to(\mb Z,\mb Z^n)$ in $\Alg{n}$.
Thus the fiber functor $\Alg{n}\to\Aff$ is an equivalence, and the category of discriminant algebras $\Alg{n}\to\Alg{2}$ is equivalent to the category of splittings $\Aff\to\Alg{2}$ sending each ring $R$ to a quadratic $R$-algebra with discriminant form isomorphic to multiplication $R^{\otimes 2}\simto R$.
Such a splitting is determined up to unique isomorphism by where it sends $\Z$, and the only \'etale quadratic $\Z$-algebra up to isomorphism is $\Z^2$, so $(R,A)\mapsto R^2$ is the unique discriminant algebra $\Alg{n}\to\Alg{2}$ up to unique isomorphism.

\subsection{A helpful lemma}

In higher ranks, not every algebra can be uniquely expressed as the base change of a universal algebra defined over $\Z$.
However, we can still show that over base rings similar to $\Z$, any two discriminant algebras must agree:

%

\begin{lemma}\label{disc-to-iso}
 Let $R$ be a ring in which $2$ is a unit or prime non-zerodivisor, and let $A$ and $B$ be quadratic $R$-algebras with an isomorphism $\extpower^2 A\cong \extpower^2 B$ identifying their discriminant bilinear forms $\disc_{A}$ and $\disc_{B}$. 
 Then there is a unique $R$-algebra isomorphism $A\cong B$ inducing the given isomorphism $\extpower^2 A\cong \extpower^2 B$.
\end{lemma}
\begin{proof}
 Work locally, so that we can write $A\cong R[x]/(x^2-sx+t)$ and $B\cong R[y]/(y^2-uy+v)$, with generators chosen so that the isomorphism $\extpower^2 A\cong \extpower^2 B$ matches $1\wedge x$ with $1\wedge y$.
 (The property of $2$ being a unit or prime non-zerodivisor is preserved under localization.)
 Then we must show that there is a unique $R$-algebra isomorphism $A\to B$ of the form $x\mapsto y+c$ for some $c\in R$.
 If $x\mapsto y+c$ is to be an isomorphism, we must have $\trace_A(x) = \trace_B(y+c)$, or $s = u+2c$, so since $2$ is a non-zerodivisor there is at most one such $c$.
 
 For existence of a suitable $c$, note that the identification of discriminants gives us the equation $s^2-4t = u^2-4v$ in $R$.
 Let $k=s-u$, so that 
 \begin{align*}
  s^2-4t &= (s-k)^2-4v\\
  &= s^2 - 2sk + k^2 - 4v\\
  k^2 &= -4t + 2sk + 4v.
 \end{align*}
 Then $k^2$ is divisible by $2$, so $k$ must be as well because $2$ is either prime or a unit.
 Then letting $c$ be such that $k=2c$, we have obtained the desired $c$ for which $s=u+2c$.

 Now we show that $x\mapsto y+c$ is in fact an $R$-algebra isomorphism.  
 First, it is well-defined, since 
 \begin{align*}
  (y+c)^2 - s(y+c) + t &= y^2 + 2cy + c^2 - sy - sc + t\\
  &= (uy-v) + 2cy + c^2 - sy - sc + t\\
  &= (u-s+2c)y + (c^2 - sc + t-v).
 \end{align*}
 Now $s = u+2c$ by construction, so the coefficient of $y$ vanishes.
 But from the equation identifying discriminants, we also obtain
 \begin{align*}
  s^2 - 4t &= (s-2c)^2 - 4v \\
  &= s^2 - 4sc + 4c^2 - 4v\\
  0 &= 4(t-sc+c^2 - v).
 \end{align*}
 So since $2$ is a non-zerodivisor, we have $c^2-sc+t-v=0$ as well.
 Therefore the map $A\to B\colon x\mapsto y+c$ is well-defined, and it is bijective since the $R$-module basis $\{1,x\}$ for $A$ is carried to an $R$-module basis $\{1,y+c\}$ for $B$. 
\end{proof}


It is easy to find counterexamples to \cref{disc-to-iso} if we drop the hypotheses on $R$---for example, $\F{4}$ and $\F{2}^2$ have isomorphic discriminants over $\F{2}$---but note that it is not enough to assume that $2$ is a non-zerodivisor:
\begin{example}
 Let $R=\Z[\sqrt{5}]$, and define two quadratic $R$-algebras
 \[A \coloneqq  R[x]/(x^2-x-1),\qquad B \coloneqq  R[y]/(y^2-\sqrt{5}y).\]
 The isomorphism $\extpower^2 A \cong \extpower^2 B$ identifying basis elements $1\wedge x$ and $1\wedge y$ identifies the two discriminants: \[
 \disc_{A}(1\wedge x) = (-1)^2 - 4(1)(-1) = 5 = (-\sqrt{5})^2 - 4(1)(0) = \disc_{B}(1\wedge y).
\]
However, this isomorphism $\extpower^2 A\cong \extpower^2 B$ is not induced by an algebra isomorphism $A\cong B$, because such an isomorphism would have to send $x\mapsto y+c$ for some $c\in R$, and then we would have
\begin{align*}
 0 &= (y+c)^2 - (y+c) - 1\\
 &= y^2 + (2c-1)y + (c^2-c-1)\\
 &= (2c-\sqrt{5}-1)y + (c^2-c-1),
\end{align*}
but $1+\sqrt{5}$ is not divisible by $2$ in $R$, so no such $c$ can exist.
\end{example}


An immediate consequence of \cref{disc-to-iso} is the main result of this subsection:

\begin{corollary}\label{universal}
 Let $n$ be a natural number, and let $D,D'\colon \Alg{n}\to\Alg{2}$ be two discriminant algebras, and let $(R,A)\in\Alg{n}$ be such that $2$ is a unit or prime non-zerodivisor in $R$.
 Then there is a unique isomorphism $D(R,A)\cong D'(R,A)$ descending to the composite isomorphism $\extpower^2D(R,A)\cong \extpower^n A\cong \extpower^2 D'(R,A)$.
\end{corollary}

\subsection{Proof for $n=2$}
The proofs for $n=2$ and $n=3$ will both proceed as follows: first we identify a suitably universal case in which \cref{universal} applies, and then we use base changes to get a choice-free natural isomorphism $D\cong D'$.

For the case $n=2$ we are showing that every endomorphism of $\Alg{2}$ over $\cat{Disc}$, as a morphism of CFGs over $\Aff$, is uniquely isomorphic to every other (e.g.\ the identity).

Indeed, let $D\colon\Alg{2} \to \Alg{2}$ be a discriminant algebra.
Consider the quadratic algebra $(R_1, A_1)\in\Alg{2}$ given by
\[R_1 = \mb Z[s_1,t_1],\quad A_1 = R_1[x_1]/(x_1^2-s_1x_1+t_1).\]
Since $2\in R_1$ is a prime non-zerodivisor, \cref{universal} applies and we find that there is a unique isomorphism $D(R_1,A_1)\cong A_1$ agreeing with the isomorphism we already have on their exterior powers.

Furthermore, morphisms to $(R_1,A_1)$ in $\Alg{2}$ have a nice interpretation: if $R$ is any ring and $A$ is a quadratic $R$-algebra, then a map $(R,A)\to(R_1,A_1)$ corresponds to a ring homomorphism $R_1\to R$ (i.e.\ a pair of elements $s,t\in R$) and an isomorphism $R\otimes_{R_1} A_1 \cong A$ (i.e.\ an $R$-algebra generator $x\in A$ such that $x^2-sx+t=0$).
So each choice of $R$-algebra generator $x\in A$ corresponds to a morphism $f_x\colon (R,A)\to(R_1,A_1)$.

Now given a choice of $R$-algebra generator $x$ of $A$, we obtain an isomorphism $D(R,A)\cong A$ via base change along $f_x$ from the canonical isomorphism $D(R_1,A_1)\cong A_1$.  
However, this isomorphism $D(R,A)\cong A$ may in principle depend on our choice of generator $x\in A$.

To show that there is no such dependence, we introduce a ring and algebra just slightly bigger than $(R_1,A_1)$.
Define
\[R_2 = R_1[r_0,u_0,u_0^{-1}],\quad A_2 = R_2[x_1]/(x_1^2-s_1x_1+t_1).\]
We distinguish two separate algebra generators for $A_2$ over $R_2$: $x_1$ by itself does the job, and so does $x_2\coloneqq u_0x_1+r_0$, since $u_0$ is a unit in $R_2$. 
We thus obtain two morphisms $(R_2,A_2)\rightrightarrows (R_1,A_1)$ corresponding to $x_1$ and $x_2$; call these morphisms $\pi_1$ and $\pi_2$.

\begin{lemma}\label{generator-change}
 The two morphisms $\pi_1,\pi_2\colon(R_2,A_2)\rightrightarrows (R_1,A_1)$ express $(R_2,A_2)$ as the categorical product of $(R_1,A_1)$ with itself in $\Alg{2}$.
 In other words, given any two morphisms $f,g\colon (R,A)\rightrightarrows(R_1,A_1)$, they factor uniquely as a single map $(f,g)\colon (R,A)\to(R_2,A_2)$ followed by $\pi_1,\pi_2\colon(R_2,A_2)\rightrightarrows (R_1,A_1)$.
\end{lemma}

\begin{proof}
 We show that maps $(R,A)\to (R_2,A_2)$ correspond to pairs of single generators for $A$ as an $R$-algebra.
 Indeed, given any two generators $x,y$ for $A$ over $R$, they must each be a free $R$-module generator for the free rank-$1$ module $A/R$, so there is a unique unit $u\in R^\times$ for which $y\equiv ux$ modulo $R$.
 Then there is also a unique element $r\in R$ such that $y = ux+r$.
 So the data of two single generators $x$ and $y$ for $A$ is equivalent to the data of a single generator $x$ along with an arbitrary element $r$ and unit $u$.
 But this is exactly the data captured by a morphism $(R,A)\to(R_2,A_2)$.
\end{proof}

Next we note that since $R_2$ is the localization of a polynomial ring over $\mb Z$, \cref{universal} applies, and we obtain a unique isomorphism $D(R_2,A_2)\cong A_2$ agreeing with the identification of their determinant bundles.  
In particular, the two base changes of the canonical isomorphism $D(R_1,A_1)\cong A_1$ via the underlying ring homomorphisms of $\pi_1$ and $\pi_2$ give the same isomorphism $D(R_2,A_2)\cong A_2$.

Then given any two morphisms $f,g\colon(R,A)\rightrightarrows (R_1,A_1)$ in $\Alg{2}$, we may use each of $f$ and $g$ to define an isomorphism $D(R,A)\cong A$, by base changing the canonical isomorphism $D(R_1,A_1)\cong A_1$ along the underlying ring homomorphisms of $f$ and $g$.
By the lemma, these base changes can also be accomplished by first base changing along $\pi_1$ or $\pi_2$ to an isomorphism $D(R_2,A_2)\cong A_2$, and then base changing along $(f,g)\colon (R,A)\to (R_2,A_2)$.
But then the resulting two isomorphisms $D(R,A)\cong A$ are equal, since they are the base changes of the same isomorphism $D(R_2,A_2)\cong A_2$.
Thus any two choices of $R$-algebra generator for $A$ give rise to the same isomorphism between $D(R,A)\cong A$.

Then for an arbitrary (not necessarily monogenic) quadratic $R$-algebra $A$, we may define a choice-free isomorphism $D(R,A)\simto A$ by building it locally, and then gluing.

This completes the proof: We know that there is just one possible choice of isomorphism $D(R_1,A_1)\cong A_1$, and that for a monogenic quadratic $R$-algebra $A$, there is exactly one isomorphism $D(R,A)\cong A$ that is compatible with the base changes from every $(R,A)\to (R_1,A_1)$.
Furthermore, for a general quadratic $R$-algebra $A$, there is exactly one isomorphism $D(R,A)\cong A$ that is compatible with the base changes to every $(R_r, A_r)\to(R,A)$ for which $(R_r,A_r)$ is monogenic.  
Naturality of this isomorphism in general then follows from uniqueness.

\subsection{Proof for $n=3$}
We use the same approach for rank-$3$ algebras as for those of rank $n=2$: identify a cubic algebra over a simple enough ring that every cubic algebra is locally a base change of it.
Not every cubic algebra is locally monogenic, but it does locally have a basis containing its multiplicative identity; in other words, if $A$ is a cubic $R$-algebra then $A/R$ is a locally free $R$-module of rank $2$.
(See, for example, \cite[Lemma 2.3]{Bie15} for a proof that if $A$ is a rank-$n$ $R$-algebra then $A/R$ is locally free of rank $n-1$.)

We thus consider the structure of a cubic $R$-algebra $A$ for which the $R$-module $\dot A \coloneqq A/R$ is free of rank $2$ as an $R$-module.
Each choice $\{\dot x, \dot y\}$ of basis for $\dot A$ lifts uniquely to a basis $\{1,x,y\}$ of $A$ for which $xy\in R$. 
Then associativity requires that multiplication in $A$ be of the following form, for some $a,b,c,d\in R$:
\begin{equation}\label{cubiclaw}
\begin{aligned}
 xy &= bc \\
 x^2 &= -bd + ax + by\\
 y^2 &= -bc + cx + dy,
 \end{aligned}
\end{equation}
and conversely, any choice of $a,b,c,d\in R$ gives rise to a cubic $R$-algebra structure on $R\langle 1,x,y\rangle$ with multiplication defined by \eqref{cubiclaw}.
(See the proof of Proposition 4.2 in \cite{Gan02}, which is given over $\Z$ but works over any base ring.)

So now define an object $(R_1,A_1)$ of $\Alg{3}$ by $R_1 = \mb Z[a_1,b_1,c_1,d_1]$ and
\begin{align*}
 A_1 = R_1[x_1,y_1]/\bigl(x_1y_1&-(b_1c_1),\\(x_1)^2&-(-b_1d_1+a_1x_1+b_1y_1),\\(y_1)^2&-(-b_1c_1+c_1x_1+d_1y_1)\bigr).
\end{align*}
Morphisms $(R,A)\to(R_1,A_1)$ in $\Alg{3}$ correspond bijectively to choices of $R$-module basis for $A/R$.
And since $2$ is a prime non-zerodivisor of $R_1$, given any two discriminant algebras $D,D'\colon \Alg{3}\to\Alg{2}$ there is exactly one $R_1$-algebra isomorphism $D(R_1,A_1)\cong D'(R_1,A_1)$ that induces the composite isomorphism $\extpower^2D(R_1,A_1)\cong\extpower^3A_1\cong \extpower^2 D'(R_1,A_1)$.
So for each $R$-module basis of $\dot A$, we obtain by base change an isomorphism $D(A,R)\cong D'(A,R)$.

The proof now continues as in the $n = 2$ case: we will
prove that this isomorphism is independent of choice of basis by considering two bases $\{\dot x, \dot y\}$ and $\{\dot x',\dot y'\}$ for $\dot A$.
In this case there is a matrix $(\begin{smallmatrix}e & f\\g & h\end{smallmatrix})$ in $\mathop{GL}_2(R)$ such that
\[\begin{pmatrix} \bar x' \\ \bar y'
\end{pmatrix} = \begin{pmatrix}e & f\\g & h
\end{pmatrix}\begin{pmatrix}\bar x\\\bar y
\end{pmatrix}.
\]
Then we expand $(R_1,A_1)$ to a new ring $(R_2,A_2)$ by setting
\[R_2 = R_1\left[e_0, f_0, g_0, h_0, (e_0h_0-f_0g_0)^{-1}\right]\]
and $A_2 = R_2\otimes_{R_1}A_1$.
We again have two distinguished maps from $(R_2,A_2)$ to $(R_1,A_1)$, corresponding to the two $R_2$-module bases $(\bar x_1,\bar y_1)$ and $(\bar x_2,\bar y_2)\coloneqq (e_0\bar x_1+f_0\bar y_1, g_0\bar x_1 + h_0\bar y_1)$ for $A_2/R_2$.
And once again, $(R_2, A_2)$ is the categorical product of $(R_1,A_1)$ with itself in $\Alg{3}$, by a simple check that maps $(R,A)\to (R_2,A_2)$ correspond to pairs of bases for $\dot A$.
And since $R_2$ is again a localization of a polynomial ring over $\mb Z$, \cref{universal} applies, and the same argument as in the $n=2$ case shows that if $A/R$ is free then the isomorphism $D(R,A)\cong D'(R,A)$ is independent of choice of basis.  
Then we can glue these isomorphisms to obtain a unique isomorphism $D(R,A)\cong D'(R,A)$ for arbitrary $A$ cubic over arbitrary $R$; the uniqueness guarantees naturality.

One might expect to be able to apply the same type of argument to arbitrary rank $n$, but it is not clear how to find the analogues of $R_1$ and $A_1$ for which $2$ is a unit or prime non-zerodivisor in $R_1$, and for which every rank-$n$ algebra is locally a base change of $A_1$.
Melanie Matchett Wood's parameterization of quartic algebras \cite{Wood12} or a similar generalization of Manjul Bhargava's parametrizations of quintic $\Z$-algebras \cite{BhargThIII} may allow such an approach to work for $n$ up to $5$, but each higher rank would be a hard-won battle without a more general approach.

\section{Isomorphism with Rost's discriminant algebra}

As a consequence of \cref{uniqueness}, there is a canonical isomorphism between the discriminant algebra $\discalg$ of \cref{existence} and any other for rank $3$.
In \cite{Rost}, Markus Rost defines such a discriminant algebra for rank $3$, and we exhibit here the resulting isomorphism with $\discalg$ as a simpler version of the analysis for Loos's discriminant algebra to follow.


We begin by reviewing the definition of Rost's discriminant algebra from \cite{Rost}.

\begin{definition}
Let $R$ be a ring and let $A$ be a rank-$3$ $R$-algebra. 
Define an $R$-module $K(A)$ as follows: First let $\dot A$ be the quotient $R$-module $A/R$.
Because $A$ has rank $3$, the function
\[
A\to R: a\mapsto s_1(a^2) - s_2(a)
\]
depends only on $\dot a$, the equivalence class of $a$ modulo $R$, and thus defines a quadratic form $q_A$ on $\dot A$.

Let $C(q_A)$ be the Clifford algebra of $q_A$, the quotient of the tensor algebra $\bigoplus_{n=0}^\infty \dot A^{\otimes n}$ by the two-sided ideal generated by elements of the form $\dot a \otimes \dot a - q_A(\dot a)$.
The Clifford algebra retains a $\Z/2\Z$-grading; we let $K(A) = C_0(q_A)$ be the even-graded part.
\end{definition}

\begin{remark}
Given any quadratic form $q$ on a locally-free rank-$2$ $R$-module $M$, the even Clifford algebra $C_0(q)$ is also locally free of rank $2$: 
the construction commutes with localization, and if $M$ has $R$-basis $\{\theta_1,\theta_2\}$, then $C_0(q)$ has basis $\{1, \theta_1\otimes\theta_2\}$, since $\theta_2\otimes\theta_1 = - \theta_1\otimes \theta_2 + q(\theta_1 + \theta_2) - q(\theta_1) - q(\theta_2).$ 
Thus we can present $C_0(q)$ as the quotient of $R\oplus M^{\otimes 2}$ by the $R$-submodule generated by elements of the form $m \otimes m - q(m)$.
In particular, the ring homomorphism $R\to C_0(q)$ fits into a short exact sequence $0 \to R \to C_0(q) \to \extpower^2 M \to 0$ of $R$-modules, where the right-hand map sends each $m_1\otimes m_2\mapsto m_1\wedge m_2$.
Thus, in case $A$ is a rank-$3$ $R$-algebra, we have (via the isomorphism $\extpower^2 \dot A \to \extpower^3 A$ sending $\dot a\wedge \dot b\mapsto 1\wedge a\wedge b$) a short exact sequence of $R$-modules
\[
0 \to R \to K(A) \to \extpower^3 A\to 0
\]
and $K(A)$ is a quadratic $R$-algebra.
\end{remark}


\begin{lemma}\label{Rost-module-iso}
Let $R$ be a ring and let $A$ be a rank-$3$ $R$-algebra. 
There is a unique $R$-module homomorphism $K(A)\to \discalg(A)$ sending $1$ to $1$ and $\dot a \otimes \dot b$ to $s_1(ab)-\gamma(1, b, a)$. 
This homomorphism is a morphism of extensions of $\extpower^3 A$ by $R$, and thus an $R$-module isomorphism.
\end{lemma}

\begin{proof}
Uniqueness is guaranteed because such elements generate $K(A)$ as an $R$-module.
 As for existence, the map $A\times A\to \discalg(A)$ sending $(a,b)$ to $s_1(ab)-\gamma(1, b, a)$ is $R$-bilinear, and sends elements of the form $(a,r)$ to $s_1(ar) - \gamma(1,r,a) = r s_1(a) - r \gamma(1,1,a) = 0$, and similarly sends elements of the form $(r,b)\mapsto 0$.
 Thus this function extends and descends to an $R$-module homomorphism $\dot A \otimes \dot A\to \discalg(A)$.  
 Together with the assignment $1\mapsto 1$, we obtain an $R$-linear homomorphism $R\oplus \dot A^{\otimes 2}\to \discalg(A)$.  
 This homomorphism sends elements of the form $\dot a \otimes \dot a$ to $s_1(a^2) - \gamma(1,a,a) = s_1(a^2) - s_2(a) = q(\dot a)$, so it descends to a homomorphism $K(A)\to \discalg(A)$ of
$R$-modules. 
 As suggested in the statement, this homomorphism fits into a map of short exact sequences:
\[
\xymatrix{
0 \ar[r] & R \ar[r]\ar@{=}[d] & K(A)\ar[r]\ar[d] & \extpower^3 A\ar@{=}[d] \ar[r] & 0\\
0 \ar[r] & R \ar[r] & \discalg(A)\ar[r] & \extpower^3 A \ar[r] & 0
}
\]
Commutativity of the left-hand square follows from the fact that $K(A)\to \discalg(A)$ sends $1\mapsto 1$; 
commutativity of the right-hand square can be checked on a general element of the form $\dot a\otimes \dot b$, which is sent to $1\wedge a\wedge b$ in $\extpower^3 A$.  
Its image in $\discalg(A)$ is $s_1(ab) - \gamma(1,b,a)$, which is sent to $0 - 1\wedge b\wedge a = 1\wedge a\wedge b$ as desired. 
Thus this homomorphism is an $R$-module isomorphism $K(A)\simto \discalg(A)$.
\end{proof}


Thus we have two quadratic $R$-algebras with a canonical isomorphism between their underlying modules.
However, this module isomorphism is not generally an $R$-algebra isomorphism.  
In fact, the isomorphism $\extpower^2 K(A) \cong K(A)/R \cong \extpower^3 A$ does not even identify the two discriminant forms on $K(A)$ and $\extpower^3 A$.  
Hence, Rost continues the definition of his discriminant algebra by ``shifting'' the multiplication on $K(A)$ so that the isomorphism $\extpower^2 K(A)\cong \extpower^3 A$ identifies the discriminant forms; we review this process and then show that the resulting algebra is isomorphic to $\Delta(A)$.

\begin{definition}
Let $R$ be a ring, let $B$ be a quadratic $R$-algebra, and let $\dot B = B/R$ be the quotient $R$-module.
Given a bilinear form $\eps\colon  \dot B\times\dot B \to R$, we may define a new multiplication $\star$ on $B$ in the following way:
\[
b_1\star b_2 = b_1b_2 - \eps(\dot b_1, \dot b_2).
\]
This new multiplication has the same two-sided identity element as the original multiplication on $B$, so it is automatically commutative and associative and defines a new $R$-algebra structure on the $R$-module $B$. 
This new $R$-algebra is called {\em the shift of $B$ by $\eps$} and denoted $B+\eps$.
\end{definition}

\begin{remark}
In \cite{Rost}, the shift of $B$ by $\eps$ is defined so that the new product of $b_1$ and $b_2$ is $b_1b_2+\eps(\dot b_1,\dot b_2)$.  
We instead follow the usage by Loos in \cite{LoosDiscAlg}, which is defined so that traces, norms, and discriminants transform as follows:
\begin{align*}
\trace_{B+\eps}(b) &= \trace_B(b)\\
\norm_{B+\eps}(b) &= \norm_B(b) + \eps(\dot b,\dot b)\\
\disc_{B+\eps}(1\wedge b_1,1\wedge b_2) &= \disc_{B}(1\wedge b_1,1\wedge b_2) - 4\eps(\dot b_1,\dot b_2). \end{align*}
\end{remark}


\begin{definition}
Let $R$ be a ring and let $A$ be a rank-$3$ $R$-algebra. 
Then the discriminant bilinear form $\disc_{A}$ on $\extpower^3 A$ may be regarded as a bilinear form on $K(A)/R$.
Then the {\em Rost discriminant algebra of $A$} is $D(A) := K(A)-{\disc_{A}}$, the shift of $K(A)$ by $-\disc_{A}$.
\end{definition}

Then the $R$-module isomorphism between $K(A)$ and $\discalg(A)$ becomes an $R$-algebra isomorphism between $D(A)$ and $\discalg(A)$.

\begin{theorem}\label{norm-in-delta-is-equal-to-Rost}
Let $R$ be a ring and let $A$ be a cubic $R$-algebra.  Then the $R$-module isomorphism $D(A) = K(A)\to \discalg(A)$ sending $1$ to $1$ and $\dot{a}\otimes\dot{b}$ to $s_1(ab)-\gamma(1,b,a)$ is an $R$-algebra isomorphism $D(A)\simto \discalg(A)$.
\end{theorem}

 The proof proceeds by showing that for all $a,b\in A$, the traces and norms of $\dot{a}\otimes\dot{b}$ in $D(A)$ and $s_1(ab)-\gamma(1,b,a)$ in $\Delta(A)$ agree, which we verify in \cref{misc}.
 Then the fact that the $R$-module isomorphism $D(A)\cong\discalg(A)$ is an algebra isomorphism is due to the following lemma, since elements of the form $\dot a\otimes \dot b$ generate $\dot D(A)$:
 
\begin{lemma}\label{trace-norm}
 Let $R$ be a ring and $D, E$ two quadratic $R$-algebras.
 Let $\phi\colon D\to E$ be an $R$-module isomorphism preserving unity.
 If there exists a subset $\Omega\subset D$ whose image generates $\dot D$ and on which $\phi$ preserves traces and norms, then $\phi$ is an algebra isomorphism.
\end{lemma}
\begin{proof}
 By working locally, we may assume that $D$ and $E$ are of the form $R[x]/(x^2-sx+m)$ and $R[y]/(y^2-ty+n)$, with the isomorphism $\phi\colon D\to E$ sending $x\mapsto y$.
 Denote the set $\Omega\subset D$ as $\{a_i+b_ix\}_{i\in I}$.
 Then we know that the ideal $(b_i:i\in I)$ is the unit ideal of $R$, and that for all $i\in I$, we have
 \[\trace_D(a_i+b_ix) = \trace_E(a_i+b_iy)\quad\text{and}\quad \norm_D(a_i+b_ix) = \norm_E(a_i+b_iy).\]
 The trace equation becomes $2a_i +b_is = 2a_i+b_it$, which implies $s=t$ since the $b_i$ generate the unit ideal.
 The norm equation becomes $a_i^2+a_ib_is+b_i^2m = a_i^2+a_ib_it+b_i^2n$, and since $s=t$ this reduces to $b_i^2m = b_i^2n$.
 The $b_i^2$ also generate the unit ideal, so $m=n$ as well.
 Therefore $\phi\colon R[x]/(x^2-sx+m)\to R[y]/(y^2-ty+n):x\mapsto y$ is an algebra isomorphism.
\end{proof}

\section{Isomorphism with Loos's discriminant algebra}

In \cite{LoosDiscAlg}, Ottmar Loos uses Rost's shifting technique to construct a discriminant algebra for algebras of arbitrary finite rank $n$.
We review his definition and show that it is also isomorphic to $\discalg$.
We begin with the case of an arbitrary quadratic form on an even-rank module:

\begin{definition}
 Let $R$ be a ring, let $M$ be a locally free rank-$n$ $R$-module with $n$ even, and let $Q\colon  M\to R$ be a quadratic form.  
 Given a bilinear form $f$ representing\footnote{We say that a bilinear form $f\colon M\otimes M\to R$ \emph{represents} a quadratic form $Q\colon M\to R$ if for each $m\in R$ we have $Q(m) = f(m\otimes m)$. Every quadratic form on a locally free finite-rank module is representable by some bilinear form; see \cref{representable}} $Q$, we will define an $R$-algebra $\loosdiscquad[f]{Q}$ with underlying $R$-module $R\oplus \extpower^{n} M$; 
 its elements $(r,u)$ are denoted $r\cdot 1_f + s_f(u)$.  
 By \cref{trace-norm}, the $R$-algebra structure is determined by the traces and norms of the $s_f(u)$ with $u$ decomposable as $u=x_1\wedge\dots\wedge x_n$; 
 these are defined to be
 \begin{align*}
  \trace(s_f(u)) &\coloneqq  \Pf\bigl( f(x_i,x_j)-f(x_j,x_i) \bigr)_{i,j=1}^{n}\\
  \norm(s_f(u)) &\coloneqq  (-1)^{\frac{n}{2} + 1}\qdet\bigl(f(x_i,x_j)\bigr)_{i,j=1}^n,
 \end{align*}
 where the Pfaffian $\Pf(A)$ of an antisymmetric matrix $A$ is the canonical square root of its determinant, and the quarter-determinant $\qdet(A)$ of an even-dimensional square matrix $A$ is the canonical quarter of $\det(A+\transpose{A}) + (-1)^{\frac{n}{2} + 1}\det(A-\transpose{A})$:
 \[
 4\qdet(A) = \det(A+\transpose{A}) + (-1)^{\frac{n}{2} + 1}\det(A-\transpose{A});
 \]
 see \cite{LoosDiscAlgQuad} for details of these matrix constructions.
\end{definition}

Loos's next step is to remove the dependence of $\loosdiscquad[f]{Q}$ on $f$.

\begin{definition}
 Let $Q\colon M\to R$ be a quadratic form on a locally free rank-$n$ $R$-module $M$ with $n$ even, and let $f, g\colon  M\times M\to R$ both represent $Q$. 
  Define a linear map $\kappa_{fg}\colon \extpower^n M\to R$ as follows.  
  Given two antisymmetric matrices $A$ and $A'$, let $t$ be an indeterminate and define $\Pi(t,A,A')$ to be the polynomial in $t$ satisfying
 \[
 \Pf(A + tA') = \Pf(A) + t \Pi(t,A,A').
 \]
 Then $\kappa_{fg}\colon \extpower^n M\to R$ is defined by 
 \[
 \kappa_{fg}(x_1\wedge\dots\wedge x_n)=\Pi\left(-2,A,A'\right),
 \]   
 where $A=\bigl(f(x_i,x_j)-f(x_j,x_i)\bigr)_{i,j=1}^n$ and $A'=\bigl(f(x_i,x_j)-g(x_i,x_j)\bigr)_{i,j=1}^n$.
 (Note that the latter is an antisymmetric matrix, because since $f$ and $g$ both represent $Q$ we have $f(x_i,x_j)+f(x_j,x_i) = Q(x_i+x_j)-Q(x_i)-Q(x_j) = g(x_i,x_j)+g(x_j,x_i)$.)
 
 Now define an $R$-module homomorphism $\loosdiscquad[f]{Q} \to \loosdiscquad[g]{Q}$ sending $1_f \mapsto 1_g$ and $s_f(u) \mapsto s_g(u) + \kappa_{fg}(u)\cdot 1_g$; 
 it is in fact an isomorphism of $R$-algebras (cf.\ \cite[Theorem 2.3]{LoosDiscAlgQuad}).  
 The family of isomorphisms constructed in this way is coherent in the sense that if $h\colon M\times M\to R$ is another bilinear form representing $Q$, the resulting triangle of isomorphisms commutes (cf.\ \cite[Theorem 1.6]{LoosDiscAlgQuad}):
 \[\begin{tikzcd}[column sep=small]
 \loosdiscquad[f]{Q} \arrow{rr}{\sim} \arrow{rd}[swap]{\sim} & &\loosdiscquad[g]{Q} \arrow{ld}{\sim}\\
 & \loosdiscquad[h]{Q} &
 \end{tikzcd}\]
 Define an $R$-module $\loosdiscquad{Q}$ as follows: First form the direct sum of $R$-modules $\bigoplus_f \loosdiscquad[f]{Q}$, where $f$ ranges over all bilinear forms representing $Q$.  
 Then $\loosdiscquad{Q}$ is the quotient of this direct sum by the submodule generated by all differences of the forms $1_f-1_g$ and $s_f(u)-\bigl(s_g(u)+\kappa_{fg}(u)\cdot 1_g\bigr)$, so that each triangle of $R$-module homomorphisms of the following form commutes:
 \[\begin{tikzcd}[column sep=small]
 \loosdiscquad[f]{Q} \arrow{rr}{\sim} \arrow{rd} & &\loosdiscquad[g]{Q} \arrow{ld}\\
 & \loosdiscquad{Q} &
 \end{tikzcd}\]
 Then in fact each $R$-module homomorphism $\loosdiscquad[f]{Q}\to\loosdiscquad{Q}$ is an isomorphism, and we can transport each $R$-algebra structure on the $\loosdiscquad[f]{Q}$ to an $R$-algebra structure on $\loosdiscquad{Q}$ that is independent of $f$.  
 (In categorical terms, we have taken the colimit of a diagram whose shape category is contractible.)
 This quadratic $R$-algebra is called the {\em Loos discriminant algebra} of $Q$.
\end{definition}

\begin{remark}\label{representable}
 Note that the construction of $\loosdiscquad{Q}$ only makes sense if $Q$ can be represented by some bilinear form $f$, so that we are not taking the colimit of an empty diagram.
 This is easy to see in case $M$ is a free $R$-module: we can always represent a quadratic form $Q\colon R^n\to R$ by the linear form $f\colon R^n\times R^n \to R$ given by
 \[f(\basis_i,\basis_j) = \begin{cases}
  B(\basis_i,\basis_j) &\text{if }i<j\\
  Q(\basis_i)&\text{if }i=j\\
  0&\text{if }i>j,
 \end{cases}\]
 where $B(x,y)\coloneqq Q(x+y)-Q(x)-Q(y)$ is the polar bilinear form of $Q$.
 Then if necessary, we could define $\loosdiscquad{Q}$ by working on patches where $M$ is free and gluing together the resulting quadratic algebras.
 This is in fact what we would do to generalize $\loosdiscquad{Q}$ to the context of a quadratic form on a finite-rank vector bundle on a scheme, but in the affine context it is not necessary: every quadratic form on a finite projective module can already by represented by a bilinear form.
 
 To see this, note that the module of bilinear forms on $M$ is given by $\Hom_{R}(M^{\otimes 2},R)$, and the module of quadratic forms is $\Hom_R(\fixpower{M}{2}{\S{2}},R)$.
 (For the latter, note that quadratic forms $M\to R$ correspond to homogeneous degree-$2$ polynomial laws by \cite[Prop.\ II.1 on p.\ 236]{PolMaps}, thus to $R$-linear functions $\Gamma^2_R(M)\to R$ by \cite[Thm.\ IV.1 on p.\ 266]{PolMaps}.
 Since $M$ is flat, we have $\Gamma^2_R(M)\cong\fixpower{M}{2}{\S{2}}$ by \cite[5.5.2.5 on p.\ 123]{deligne1973cohomologie}.)
 Since both $M^{\otimes 2}$ and $\fixpower{M}{2}{\S{2}}$ are finitely presented, both of these hom-module constructions commute with localization, so the homomorphism $\Hom_R(M^{\otimes 2},R)\to\Hom_R(\fixpower{M}{2}{\S{2}},R)$ is surjective because it is locally so.
\end{remark}


Finally, Loos defines the discriminant algebra of an even-rank algebra as a shift of the discriminant algebra of its quadratic trace $s_2$.

\begin{definition}
 Let $R$ be a ring and let $A$ be a rank-$n$ $R$-algebra with $n$ even, and let $\loosdiscquad{Q}$ be the Loos discriminant algebra of the quadratic form $Q=s_2\colon A\to R$.  
 Then the {\em Loos discriminant algebra $\loosdiscalg(A)$ of $A$ over $R$} is the shift of $\loosdiscquad{Q}$ by the bilinear form $(-1)^{\frac{n}2 + 1}\floor{n/4} \cdot \disc_{A}$:
 \[
 \loosdiscalg(A) \coloneqq \loosdiscquad{Q} + (-1)^{\frac{n}2 + 1}\floor{n/4} \cdot \disc_{A}.
 \]
\end{definition}

By \cite[Theorem 2.3(c)]{LoosDiscAlgQuad} and \cite[Lemma 3.6]{LoosDiscAlg}, this construction is a discriminant algebra in the sense of \cref{discalg}.
In the remainder of this section, we will exhibit an isomorphism of discriminant algebras $\loosdiscalg\cong\discalg$.


\begin{remark}
 Loos extends his definition of discriminant algebra to cover the odd-rank case by defining $\loosdiscalg(A)\coloneqq\loosdiscalg(R\times A)$ if $A$ is an $R$-algebra of odd rank.
 By \cite[Theorem 8.5]{Bie15}, for such an algebra $A$ we also have $\discalg(A)\cong \discalg(R\times A)$, so we obtain an isomorphism $\loosdiscalg\cong\discalg$ in the odd-rank case from the isomorphism in the even-rank case.
\end{remark}


\begin{definition}\label{def-Rtilde-Xtilde}
 Fix an even natural number $n$.  
 Define a ring $R_1$ as the polynomial ring $\Z[f_{ij}: 1\leq i<j\leq n]$ in $\binom{n}{2}$ indeterminates. 
 Let $\Xbase$ be the polynomial $\Rbase$-algebra in $n$ variables $\Rbase[x_1,\dots,x_n]$.  
 Then for each ring $R$ and elements $a_1,\dots,a_n$ in an $R$-algebra $A$ equipped with a bilinear form $f\colon A\times A\to R$, we obtain a ring homomorphism $\psi_{(f,a_1,\dots,a_n)}\colon \Rbase\to R$ sending each $f_{ij}\mapsto f(a_i,a_j)$.  
 This makes $R$, and hence $A$, into an $\Rbase$-algebra, and we also obtain an $\Rbase$-algebra homomorphism $\chi_{(f,a_1,\dots,a_n)}\colon \Xbase\to A$ sending each $x_i\mapsto a_i$.
\end{definition}

\begin{definition}\label{def-Qi-Bij-Fij}
 With $\Rbase$ and $\Xbase$ as in Definition~\ref{def-Rtilde-Xtilde}, define the following $\permgroup{n}$-invariant elements of $\Xbase^{\otimes n}$, the $n$th tensor power of $\Xbase$ over $\Rbase$:
 \begin{itemize}
  \item For each $i\in\set{n}$, define $Q_i \coloneqq e_2(x_i) = \sum_{k<\ell}\conjugate{x_i}{k}\conjugate{x_i}{\ell}$.
  \item For each $i,j\in\set{n}$, define $B_{ij}\coloneqq e_2(x_i+x_j)-e_2(x_i)-e_2(x_j) = \sum_{k\neq\ell} \conjugate{x_i}{k}\conjugate{x_j}{\ell}$.  
 \end{itemize}
 Now let $R$ be a ring and $A$ a rank-$n$ $R$-algebra, let $f\colon A\times A\to R$ be a bilinear form representing the quadratic form $Q=s_2\colon A\to R$, and let $B(a,a')=Q(a+a')-Q(a)-Q(a')=f(a,a')+f(a',a)$ be the polar form of $Q$.  
 Choose elements $a_1,\dots,a_n\in A$, and let $\chi=\chi_{(f,a_1,\dots,a_n)}$ be the homomorphism $\Xbase\to A$ of Definition~\ref{def-Rtilde-Xtilde}. 
 Then obtain the composite homomorphism
 \[
 \ferrand_{A/R}\circ \chi^{\otimes n}\colon  \fixpower{\Xbase}{n}{\permgroup n} \to \fixpower{A}{n}{\permgroup{n}}\to R,
 \]
 which sends $Q_i \mapsto Q(a_i)$ and $B_{ij} \mapsto B(a_i,a_j)$ for all $i,j\in\set{n}$.  
 Now if we also define elements $F_{ij}\in\fixpower{\Xbase}{n}{\permgroup{n}}$ for each $i,j\in\set{n}$ by
 \[
 F_{ij} = \begin{cases}
  f_{ij} &\text{if }i<j\\
  Q_i &\text{if }i=j\\
  B_{ij}-f_{ji} &\text{if }i>j,
 \end{cases}
 \]
 we find that 
 \[
 \phi_{A/R}\circ \chi^{\otimes n}\colon F_{ij}\mapsto \begin{cases}
  f(a_i,a_j) &\text{if }i<j\\
  Q(a_i)=f(a_i,a_i) &\text{if }i=j\\
  B(a_i,a_j)-f(a_j,a_i)=f(a_i,a_j) &\text{if }i>j.
 \end{cases}
 \]
 So in all cases, $F_{ij}\mapsto f(a_i,a_j)$.  
 We will denote by $F$ the $n\times n$-matrix of elements of $\fixpower{\Xbase}{n}{\permgroup{n}}$ whose $i,j$th entry is $F_{ij}$.
\end{definition}

\begin{remark}\label{rmk-homom-to-discalg}
 For each choice of representative $f$ of $Q$ and elements $a_1,\dots,a_n\in A$, we may also use the homomorphism $\chi=\chi_{f,a_1,\dots,a_n}\colon \Xbase\to A$ of Definition~\ref{def-Qi-Bij-Fij} to obtain a composite $\fixpower{\Xbase}{n}{\altgroup{n}}\to\fixpower{A}{n}{\altgroup{n}}\to\discalg(A)$ that also sends each $F_{ij}$ to $f(a_i,a_j)$.  
 Thus we may construct elements of $\discalg(A)$ as the images of $\altgroup{n}$-invariant elements of $\Xbase^{\otimes n}$.
\end{remark}


\begin{lemma}\label{lemma-Sigma-existence}
 Let $\Rbase$ and its algebra $\Xbase$ be as in Definition~\ref{def-Rtilde-Xtilde}, and $F$ the $n\times n$-matrix of \cref{def-Qi-Bij-Fij}. 
 There is a unique element $\Sigma_F$ of $\fixpower{\Xbase}{n}{\altgroup{n}}$ such that
 \[
 2\Sigma_F = \det\bigl(\conjugate{x_i}{j}\bigr)_{i,j=1}^n + \Pf(F-\transpose{F}).
 \]
\end{lemma}

We relegate the proof of this lemma to \cref{misc}.

\begin{definition}\label{notation-sigma-f}
 Let $R$ be a ring and $A$ be a rank-$n$ $R$-algebra with $n\geq 2$ and even.
 Let $f$ be a bilinear form representing the quadratic form $Q=s_2\colon A\to R$, and let $a_1,\dots,a_n\in A$.
 Let $\Xbase$ be as in Definition~\ref{def-Rtilde-Xtilde}, with the homomorphism $\chi=\chi_{(f,a_1,\dots,a_n)}\colon \Xbase\to A$ of Definition~\ref{def-Qi-Bij-Fij} and the element $\Sigma_F\in\fixpower{\Xbase}{n}{\altgroup{n}}$ of Lemma~\ref{lemma-Sigma-existence}. 
 Define $\sigma_f(a_1,\dots,a_n)\in\discalg(A)$ as the image of $\Sigma_F$ under the homomorphism $\fixpower{\Xbase}{n}{\altgroup{n}}\to\fixpower{A}{n}{\altgroup{n}}\to \discalg(A)$ as in Remark~\ref{rmk-homom-to-discalg}.
\end{definition}

\begin{remark}
By construction, $2\sigma_f(a_1,\dots,a_n)\in\discalg(A)$ is always equal to 
\[
\gamma(a_1,a_2,\dots,a_n) - \gamma(a_2,a_1,\dots,a_n) + \Pf(f(a_i,a_j)-f(a_j,a_i))_{ij=1}^n.
\]
\end{remark}

\begin{theorem}\label{loos-mod-iso}
 Let $R$ be a ring and let $A$ be a rank-$n$ $R$-algebra with $n\geq 2$ and even, with the quadratic form $Q=s_2\colon A\to R$. 
 Then there is a unique $R$-module isomorphism $\loosdiscquad{Q}\to \discalg(A)$ sending $1$ to $1$ and $s_f(a_1\wedge\dots\wedge a_n)$ to $\sigma_f(a_1,\dots,a_n)$ for each tuple $(a_1,\dots,a_n)\in A^n$ and each bilinear representative $f$ of $Q$.
\end{theorem}

\begin{proof}
 Uniqueness is assured, since $1$ and elements of the form $s_f(a_1\wedge\dots\wedge a_n)$ generate $\loosdiscquad{Q}$.  
 To demonstrate existence, we show that for each $f$ representing $Q$, the assignment $1_f\mapsto 1$ and $s_f(a_1\wedge\dots\wedge a_n)\mapsto \sigma_f(a_1,\dots,a_n)$ extends to an $R$-module homomorphism $\loosdiscquad[f]{Q}\to \discalg(A)$, and then that for any pair of representatives $f$ and $g$ of $Q$, the homomorphisms $\loosdiscquad[f]{Q}\to \discalg(A)$ and $\loosdiscquad[g]{Q}\to\discalg(A)$ commute with the isomorphism $\loosdiscquad[f]{Q}\to\loosdiscquad[g]{Q}$.  
 These amount to the following claims about the $\sigma_f$:
 \begin{enumerate}
  \item Multilinearity: $\sigma_f(a_1,\dots,a_n)$ is $R$-linear in each $a_k$.
  \item Alternation: $\sigma_f(a_1,\dots,a_n)=0$ if $a_i=a_j$ for any $i\neq j$.
  \item Compatibility: $\sigma_f(a_1,\dots,a_n) = \sigma_g(a_1,\dots,a_n) + \kappa_{fg}(a_1\wedge\dots\wedge a_n)$.
 \end{enumerate}
 
 {\em Proof of multilinearity.}
  Let $\chi=\chi_{(f,a_1,\dots,a_n)}\colon \Xbase\to A$ be the ring homomorphism of Definition~\ref{def-Rtilde-Xtilde}.  
  We know that under $\phi_{A/R}\circ \chi^{\otimes n}\colon  \fixpower{\Xbase}{n}{\permgroup{n}}\to R$, the images of $f_{ij}$ and $B_{ij}$ are $f(a_i,a_j)$ and $f(a_i,a_j)+f(a_j,a_i)$, respectively, and hence vary $R$-linearly with $a_i$ and $a_j$.  
  Then for each $k\in\set{n}$, since every term of $\Pf(F-\transpose{F})$ contains exactly one factor of the form $\conjugate{x_k}{j}$ for some $j$, or $f_{ik}$ or $f_{kj}$ for some $i<k$ or $j>k$, the image of $\Pf(F-\transpose{F})$ varies $R$-linearly with $a_k$. 
 Similarly, every term in $\det\bigl(\conjugate{ x_i}{j}\bigr)_{ij=1}^n$ contains exactly one factor of the form $\conjugate{ x_k}{j}$ for some $j$, and hence varies $R$-linearly with $a_k$.
 The image of $\Sigma_F$ under $\chi$, which is half of their sum, must then also vary $R$-linearly with each $a_k$.
 Thus $\sigma_f$ extends to an $R$-linear map $A^{\otimes n}\to \discalg(A)$.
 
 {\em Proof of alternation.}  
 Suppose $a_k=a_\ell$ for some pair $k\neq \ell$.  
 Then the homomorphism $\Xbase\to A$ is unchanged under the automorphism of $\Xbase$ transposing $ x_k$ and $ x_\ell$; 
 therefore it factors through the quotient $\Xbase' = \Xbase/( x_k -  x_\ell)$.  
 But $\det(\conjugate{ x_i}{j})$ and $\Pf(F-\transpose{F})$ both vanish under the quotient map $\Xbase^{\otimes n}\to\Xbase'^{\otimes n}$. 
 Then $\Sigma_F$ also maps to $0$ in $\Xbase'^{\otimes n}$, and so $\sigma_f(a_1,\dots,a_n)=0$.  
 Therefore $\sigma_f$ descends to an $R$-linear map $\extpower^n A \to \discalg(A)$, and together with the assignment $1_f\mapsto 1$ we obtain an $R$-linear map $\loosdiscquad[f]{Q}\cong R\oplus \extpower^n A\to \discalg(A)$ sending each $s_f(x_1\wedge\dots\wedge x_n)$ to $\sigma_f(x_1,\dots,x_n)$.
 
 {\em Proof of compatibility.} 
 Let $f$ and $g$ be two bilinear forms representing $Q$.  
 For each choice of $a_1,\dots,a_n\in A$, we obtain two ring homomorphisms $\Rbase\to R$ and two homomorphisms $\Xbase\to A$; 
 we may represent this situation by now making $R$ into a $\Rbasetwo \coloneqq \Rbase \otimes_\Z \Rbase \cong \Z[f_{ij}, g_{ij}: 1\leq i<j\leq n]$-algebra; then letting $\Xbasetwo \coloneqq \Rbasetwo\otimes_\Rbase \Xbase \cong \Rbasetwo[ x_1,\dots, x_n]$ we have an $\Rbasetwo$-algebra homomorphism $\Xbasetwo\to A$ sending each $ x_i \mapsto a_i$, $f_{ij}\mapsto f(a_i,a_j)$, and $ g_{ij}\mapsto g(a_i,a_j)$. 
  Then we can build matrices $F$ and $G$ with entries in $\fixpower{\Xbasetwo}{n}{\altgroup{n}}$, and thus elements 
 \begin{align*}
  \Sigma_F &= \frac12\bigl(\det(\conjugate{ x_i}{j}\bigr)_{i,j} + \Pf(F-\transpose{F}))\mapsto \sigma_f(a_1,\dots,a_n)\\
  \Sigma_G &= \frac12\bigl(\det(\conjugate{ x_i}{j}\bigr)_{i,j} + \Pf(G-\transpose{G}))\mapsto \sigma_g(a_1,\dots,a_n)
 \end{align*}
 of $\fixpower{\Xbasetwo}{n}{\altgroup{n}}$.
 Then $\Sigma_F-\Sigma_G\mapsto \sigma_f(a_1,\dots,a_n)-\sigma_g(a_1,\dots,a_n)$; 
 we will show that it also maps to $\kappa_{fg}(a_1\wedge\dots\wedge a_n)$.  
 Indeed, $\Sigma_F-\Sigma_G$ is a difference of Pfaffians:
 \[
 \Sigma_F-\Sigma_G = \frac12\bigl(\Pf(F-\transpose{F}) - \Pf(G-\transpose{G})\bigr).
 \]
  Now $F-G$ is an antisymmetric matrix whose $ij$th entry with $i<j$ is $f_{ij} - g_{ij}$.
 Then $(G-\transpose{G}) = (F-\transpose{F}) - 2(F-G)$, so
  \begin{align*}
   \Pf(G-\transpose{G}) &= \Pf\bigl((F-\transpose{F})- 2(F-G)\bigr)\\
   &= \Pf(F-\transpose{F}) - 2\Pi(-2,F-\transpose{F}, F-G),
  \end{align*}
  so
  \begin{align*}
   \Sigma_F-\Sigma_G &= \frac12\bigl(\Pf(F-\transpose{F}) - \Pf(F-\transpose{F}) + 2\Pi(-2,F-\transpose{F},F-G)\bigr)\\
   &= \Pi(-2,F-\transpose{F},F-G).
  \end{align*}
 Now under the homomorphism $\Xbasetwo^{\otimes n}\to A^{\otimes n}$, for each $i<j$ the $ij$th entry of the matrix $F-\transpose{F}$ maps as $F_{ij}-F_{ji} \mapsto f(a_i,a_j) - f(a_j,a_i)$, and the $ij$th entry of $F-G$ maps as $F_{ij}-G_{ij} \mapsto f(a_i,a_j)-g(a_i,a_j)$.
 Thus
 \begin{align*}
  \Sigma-\Sigma' &\mapsto \Pi\left(-2,\bigl(f(a_i,a_j) - f(a_j,a_i)\bigr)_{i,j=1}^n,\bigl(f(a_i,a_j)-g(a_i,a_j)\bigr)_{i,j=1}^n\right)\\ 
  &= \kappa_{fg}(a_1\wedge\dots\wedge a_n).
 \end{align*}
 Thus the assignments $1\mapsto 1$ and $s_f(a_1\wedge\dots\wedge a_n)\mapsto \sigma_f(a_1,\dots, a_n)$ constitute a well-defined $R$-module homomorphism $\loosdiscquad{Q}\to \discalg(A)$.
 
 To show that this homomorphism is an $R$-module isomorphism, merely note that it fits into a map of short exact sequences
 \[\xymatrix{
  0\ar[r] & R \ar[r]\ar@{=}[d] & \loosdiscquad{Q} \ar[d] \ar[r] & \extpower^n A \ar[r]\ar@{=}[d] & 0 \\
  0\ar[r] & R \ar[r] & \discalg(A) \ar[r] & \extpower^n A \ar[r] & 0
}\]
 since it sends $1\mapsto 1$ and for each $a_1,\dots,a_n\in A$, we have both $s_f(a_1\wedge\dots\wedge a_n)\mapsto a_1\wedge\dots\wedge a_n$ and $\sigma_f(a_1,\dots,a_n)\mapsto a_1\wedge\dots a_n$; 
 indeed, the difference between $\Sigma_F$ and $\gamma(x_1,\dots, x_n)$ in $\Xbase^{\otimes n}$ is $\permgroup{n}$-invariant, hence $\sigma_f(a_1,\dots,a_n)$ and $\gamma(a_1,\dots,a_n)$ differ by an element of $R$ in $\discalg(A)$ and so have the same image in $\extpower^n A$, namely $a_1\wedge\dots\wedge a_n$.  
\end{proof}

\begin{theorem}\label{loos-iso}
 Let $R$ be a ring, and let $A$ be a rank-$n$ $R$-algebra with $n\geq 2$ and even. 
 The $R$-module isomorphism $\loosdiscalg(A)\cong\loosdiscquad{Q}\to\discalg(A)$ of \cref{loos-mod-iso} is an $R$-algebra isomorphism $\loosdiscalg(A)\simto \discalg(A)$.
\end{theorem}

\begin{proof}
 By \cref{trace-norm}, to show that the the $R$-module isomorphism $\loosdiscalg(A)\to \discalg(A)$ is an isomorphism of algebras, it is enough to check that it preserves traces and norms of elements of the form $s_f(a_1\wedge\dots\wedge a_n)$, since the images of these in $\loosdiscalg(A)/R\cong \extpower^n A$ are a generating set. 
 First we check that the traces of the two elements $\sigma_f(a_1,\dots,a_n)\in\discalg(A)$ and $s_f(a_1\wedge\dots\wedge a_n)\in\loosdiscalg(A)$ agree.  
 The trace of $s_f(a_1\wedge\dots\wedge a_n)$ in $\loosdiscquad{Q}$ is defined to be $\Pf\bigl(f(a_i,a_j)-f(a_j,a_i)\bigr)_{i,j=1}^n$, and this trace is preserved under the shift to $\loosdiscalg(A)$. 
 To compute the trace of $\sigma_f(a_1,\dots,a_n)$, we recall that the trace of an element of a quadratic algebra is the sum of it and its conjugate, i.e.\ the image of it under the standard involution.
 The standard involution on $\discalg(A)$ descends from the action of a transposition on the tensor factors of $\fixpower{A}{n}{\A{n}}$, so we may compute the trace of $\sigma_f(a_1,\dots,a_n)$ as the image in $\discalg(A)$ of $\Sigma_F+\tau(\Sigma_F)$, where $\tau\colon \fixpower{X_1}{n}{\A{n}}\to\fixpower{X_1}{n}{\A{n}}$ is the action of a transposition.  
 This sum is
 \begin{align*}
  \Sigma_F+\tau(\Sigma_F) &= \frac12\left(\det\bigl(\conjugate{ x_i}{j}\bigr)_{i,j=1}^n + \Pf(F-\transpose{F})\right)\\
  &\qquad + \frac12\left(-\det\bigl(\conjugate{ x_i}{j}\bigr)_{i,j=1}^n + \Pf(F-\transpose{F})\right)\\&=
  \frac12\cdot 2\,\Pf(F-\transpose{F}) = \Pf(F-\transpose{F}).
 \end{align*}
 Hence the trace of $\sigma_f(a_1,\dots,a_n)$ is the image of $\Pf(F-\transpose{F})$ in $\discalg(A)$, namely $\Pf\bigl(f(a_i,a_j)-f(a_j,a_i)\bigr)_{i,j=1}^n$ as desired.

Next, we calculate the norms of $\sigma_f(a_1,\dots,a_n)$ in $\discalg(A)$ and of $s_f(a_1\wedge\dots\wedge a_n)$ in $\loosdiscalg(A)$. 
The norm of $\sigma_f(a_1,\dots,a_n)$ may be calculated in the same manner as its trace, as the image of $\Sigma_F\cdot\tau(\Sigma_F)$ under the homomorphism $\fixpower{\Xbase}{n}{\altgroup{n}}\to\discalg(A)$.
 We obtain
\begin{align*}
  \Sigma_F\cdot\tau(\Sigma_F) &= \frac12\left(\det\bigl(\conjugate{x_i}{j}\bigr)_{i,j=1}^n + \Pf(F-\transpose{F})\right)\\
  &\qquad \cdot \frac12\left(-\det\bigl(\conjugate{x_i}{j}\bigr)_{i,j=1}^n + \Pf(F-\transpose{F})\right)\\&=
  \frac14\left(-\left(\det\bigl(\conjugate{x_i}{j}\bigr)_{i,j=1}^n\right)^2 + \Pf(F-\transpose{F})^2\right)\\
  &=\frac14\left(-\det\left(\sum_{k=1}^n\conjugate{x_i}{k}\conjugate{x_j}{k}\right)_{i,j=1}^n + \det(F-\transpose{F})\right)\\
  &=\frac14\left(-\det\bigl(e_1(x_ix_j)\bigr)_{i,j=1}^n + \det(F-\transpose{F})\right),
 \end{align*}
  and the norm of $\sigma_f(a_1,\dots,a_n)$ is therefore the image of this quantity in $\discalg(A)$.
  
 On the other hand, the norm of $s_f(a_1\wedge\dots\wedge a_n)$ in $\loosdiscquad{Q}$ is defined as
 \[
 \norm_{\loosdiscquad{Q}}\bigl(s_f(a_1\wedge\dots\wedge a_n)\bigr) = (-1)^{\frac{n}2 + 1}\qdet\bigl(f(a_i,a_j)\bigr)_{i,j=1}^n,
 \]
 which is the image of
 \begin{align*}
  (-1)^{\frac{n}2+1}\qdet(F) &= \frac14\left((-1)^{\frac{n}2+1}\det(F+\transpose{F}) + \det(F-\transpose{F})\right)\\
  &= \frac14\left((-1)^{\frac{n}2+1}\det\bigl(B_{ij}\bigr)_{i,j=1}^n + \det(F-\transpose{F})\right),
 \end{align*}
 since $F_{ij} + F_{ji} = B_{ij}$ for all $i,j\in\set{n}$.
 Now the norm of $s_f(a_1\wedge\dots\wedge a_n)$ in $\loosdiscalg(A)$ is $(-1)^{\frac{n}2+1}\floor{n/4}\disc_{A}(a_1\wedge\dots\wedge a_n,a_1\wedge\dots\wedge a_n)$ more than its norm in $\loosdiscquad{Q}$, and is hence the image of
 \begin{align*}
  (-1)^{\frac{n}2+1}\floor{\frac{n}4}\det\bigl(e_1(x_ix_j)\bigr)_{i,j=1}^n &+ \frac14\left((-1)^{\frac{n}2+1}\det\bigl(B_{ij}\bigr)_{i,j=1}^n + \det(F-\transpose{F})\right).
 \end{align*}
 Now we claim that the following identity holds in the polynomial ring $\Xbase^{\otimes n}$; 
 it is proved as Lemma~\ref{lemma-disc-identity}.
 \[
 \det\bigl(B_{ij}\bigr)_{i,j=1}^n = (1-n)\det\bigl(e_1(x_ix_j)\bigr)_{i,j=1}^n.
 \]
 Then the norm of $s_f(a_1\wedge\dots\wedge a_n)$ in $\loosdiscalg(A)$ is the image of 
 \begin{align*}
 &\quad\ \frac14\left((-1)^{\frac{n}2+1}\left(4\floor{\frac{n}4}\det\bigl(e_1(x_ix_j)\bigr)_{i,j=1}^n + \det(B_{ij})_{i,j=1}^n\right) + \det(F-\transpose{F})\right)\\
 &= \frac14\left((-1)^{\frac{n}2+1}\left(4\floor{\frac{n}4} + (1-n)\right)\det\bigl(e_1(x_ix_j)\bigr)_{i,j=1}^n + \det(F-\transpose{F})\right),
\end{align*}
and since $4\floor{n/4}=n-1 + (-1)^{\frac{n}2}$ for $n$ even, this simplifies to
\[
\frac14\left(-\det\bigl(e_1(x_ix_j)\bigr)_{i,j=1}^n + \det(F-\transpose{F})\right),
\]
 whose image in $R$ we have already shown to be the norm of $\sigma_f(a_1,\dots,a_n)$ in $\discalg(A)$, as desired.
 
 Therefore the trace and norm of $s_f(a_1\wedge\dots\wedge a_n)$ in $\loosdiscalg(A)$ and of $\sigma_f(a_1,\dots,a_n)$ in $\discalg(A)$ agree, so the $R$-module isomorphism $\loosdiscalg(A)\to\discalg(A)$ sending $1\mapsto 1$ and $s_f(a_1\wedge\dots\wedge a_n)\mapsto \sigma_f(a_1,\dots,a_n)$ is an $R$-algebra isomorphism, as desired.
\end{proof}

\appendix
\section{Miscellaneous Calculations}
\label{misc}

In this appendix, we perform various computations used to support claims made in the main text.

\subsection{Traces and norms in rank 3}
First, we show that the module isomorphism $D(A)\to\discalg(A)$ of \cref{norm-in-delta-is-equal-to-Rost} preserves traces and norms by means of the following identity:

\begin{lemma}\label{relations-for-sk}
 Let $A$ be an $R$-algebra of rank $3$ with elements $a,b\in A$.
 Then 
 \begin{align*}
 s_1(a)s_1(b) &= s_1(ab) + s_2(a+b) - s_2(a) - s_2(b)\\
 s_1(a)^2 &= s_1(a^2) + 2s_2(a)\\
s_1(a)s_1(ab)s_1(b) &= s_1(a^2b) s_1(b) +s_1(a) s_1(ab^2)   - s_1(a^2b^2)+ s_2(ab) + s_2(a)s_2(b).
\end{align*}
\end{lemma}
\begin{proof}
 It is possible to prove these with the expressions for $s_k(a+b)$ shown by Christophe Reutenauer and Marcel-Paul Sch{\"u}tzenberger in \cite{DeterminantSum}, but we present here a different argument.
 
 For each $x\in A$, $s_k(x)$ is the image under $\ferrand_{A/R}\colon \fixpower{A}{3}{\S{3}}\to R$ of the element $e_k(x)\in \fixpower{A}{3}{\S{3}}$.
 Thus it suffices to check that the corresponding identities in the $e_k(x)$ hold in $A^{\otimes 3}$.
 
 For example,
 \begin{align*}
  e_1(a)e_1(b) &= (a\otimes 1\otimes 1+\dots)(b\otimes 1\otimes 1 + \dots)\\
  &= (ab\otimes 1\otimes 1 + \dots) + (a\otimes b\otimes 1 + \dots),
 \end{align*}
 where $+\dots$ denotes the sum over all ways of forming a pure tensor out of the same tensor factors (thus the first three parenthesized expressions contain $3$ terms each, and the fourth contains $6$).
 Now $(ab\otimes 1\otimes 1 + \dots) = e_1(ab)$, and the remainder is also the difference between $e_2(a+b) = (a+b)\otimes(a+b)\otimes 1 + \dots$ and $e_2(a)+e_2(b) = (a\otimes a\otimes 1 + \dots) + (b\otimes b\otimes1 + \dots)$.
 So $e_1(a)e_1(b) = e_1(ab) + e_2(a+b)-e_2(a)-e_2(b)$.
 
 The second identity follows from the first by setting $a=b$.
 
 We handle the third identity similarly to the first.
 Expanding out $e_1(a)e_1(ab)e_1(b)$ gives
 \begin{align*}
  e_1(a)e_1(ab)e_1(b) &= (a\otimes1\otimes 1+\dots)(ab\otimes1\otimes 1+\dots)(b\otimes1\otimes 1+\dots)\\
  &= (a^2b^2\otimes1\otimes1 + \dots)\\
  &\qquad + (a^2b \otimes b \otimes 1 + \dots)\\
  &\qquad + (a\otimes ab^2 \otimes 1 + \dots)\\
  &\qquad + 2(ab\otimes ab \otimes 1 + \dots)\\
  &\qquad + (a\otimes ab \otimes b+ \dots)\\
  &= e_1(a^2b^2) + \bigl(e_1(a^2b)e_1(b)-e_1(a^2b^2)\bigr)\\
  &\qquad + \bigl(e_1(a)e_1(ab^2)-e_1(a^2b^2)\bigr)+ 2\bigl(e_2(ab)\bigr)\\
  &\qquad  + (a\otimes ab\otimes b+\dots).
 \end{align*}
 That last term also appears in the expansion of $e_2(a)e_2(b)$:
 \begin{align*}
  e_2(a)e_2(b) &= (a\otimes a\otimes 1 + \dots)(b\otimes b \otimes 1 + \dots)\\
  &= (ab\otimes ab \otimes 1 + \dots) + (a\otimes ab \otimes b + \dots)\\
  &= e_2(ab) + (a\otimes ab \otimes b+\dots).
 \end{align*}
 Hence $e_1(a)e_1(ab)e_1(b) = e_1(a^2b)e_1(b) + e_1(a)e_1(ab^2) - e_1(a^2b^2) + e_2(ab) +e_2(a)e_2(b)$.
\end{proof}


We will also use the following fact:
\begin{lemma}\label{std-involution}
Let $R$ be a ring and $D$ a quadratic $R$-algebra.
If $\tau\colon D\to D$ is an $R$-module homomorphism with the properties that
\begin{enumerate}
	\item $\tau(1)=1$, and
	\item $a\cdot\tau(a)\in R$ for all $a\in D$,
\end{enumerate}
then $\tau$ is an $R$-algebra automorphism of $D$, and for all $a\in D$, we have $\trace_D(a) = a+\tau(a)$ and $\norm_D(a) = a\cdot\tau(a)$.
\end{lemma}
\begin{proof}
 See \cite[Lemmas 2.9 and 2.13]{Voight11}.
\end{proof}
For example, if $A$ is an $R$-algebra of rank $n\geq 2$, then the automorphism of $A^{\otimes n}$ given by interchanging the first two tensor factors induces an involution on $\discalg(A)$ satisfying the hypotheses of \cref{std-involution}, so we can deduce (for example) that
\[\trace_{\discalg(A)}\bigl(\gamma(a_1,\dots,a_n)\bigr) = \gamma(a_1,a_2,\dots,a_n) +\gamma(a_2,a_1,\dots,a_n).\]
Similarly, given any quadratic form $q$ on a locally-free rank-$2$ module $M$, the even Clifford algebra $C_0(q) \cong R\oplus M^{\otimes 2} / (m\otimes m - q(m) : m\in M)$ has an involution sending $1\mapsto 1$ and elements of the form $m_1\otimes m_2$ to $m_2\otimes m_1$.
This involution satisfies the hypotheses of \cref{std-involution}, so the norm (for example) of $m_1\otimes m_2$ is
\begin{align*}
\norm_{C_0(q)}(m_1\otimes m_2) &= (m_1\otimes m_2)(m_2\otimes m_1)\\
&= m_1\otimes m_2\otimes m_2\otimes m_1\\
&= (m_1\otimes m_1) \cdot q(m_2) = q(m_1)q(m_2).
\end{align*}

Now we in a position to prove \cref{norm-in-delta-is-equal-to-Rost}, that the module isomorphism $\discalg(A)\cong K(A)$ of \cref{Rost-module-iso} becomes an \emph{algebra} isomorphism when we shift the multiplication of $K(A)$ by $-\disc_A$.

\begin{proof}[Proof of \cref{norm-in-delta-is-equal-to-Rost}]
By \cref{trace-norm}, it is enough to show that the $R$-module isomorphism $D(A)\to \discalg(A)$ preserves $1$ (which is immediate), as well as traces and norms of elements of the form $\dot a\otimes \dot b$.
Given such an element, its trace in $D(A)$ is 
\begin{align*}
\trace_{D(A)}(\dot a\otimes \dot b) &= \dot a \otimes \dot b + \dot b \otimes \dot a\\
&= (\dot a + \dot b)\otimes (\dot a + \dot b) - \dot a \otimes \dot a - \dot b \otimes \dot b\\
&= q(\dot a + \dot b) - q(\dot a) - q(\dot b)\\
&= s_1\bigl((a+b)^2\bigr) - s_2(a+b) - s_1(a^2) + s_2(a) - s_1(b^2) + s_2(b)\\
&= s_1(a^2 + 2ab + b^2) - s_1(a^2) - s_1(b^2) - \bigl(s_2(a+b) - s_2(a) - s_2(b)\bigr)\\
&= s_1(2ab) - \bigl(s_1(a)s_1(b) - s_1(ab)\bigr)\\
&= 3s_1(ab) - s_1(a)s_1(b).
\end{align*}
On the other hand, its image $s_1(ab) - \gamma(1,b,a)$ in $\discalg(A)$ has trace
\begin{align*}
\trace_{\discalg(A)}\bigl(s_1(ab) - \gamma(1,b,a)\bigr) &= 2 s_1(ab) - \gamma(1,b,a)-\gamma(1,a,b)\\
 &= 2s_1(ab) - (s_1(a)s_1(b) - s_1(ab))\\
 &= 3s_1(ab) - s_1(a)s_1(b),
\end{align*}
so the isomorphism $D(A)\to \discalg(A)$ preserves the traces of elements of the form $\dot a\otimes \dot b$.

Now the norms of $K(A)$ and $D(A)$ differ; the equation $\norm_{D(A)}(\dot a\otimes \dot b) = \norm_{\discalg(A)}\bigl(s_1(ab)-\gamma(1,b,a)\bigr)$ expands to
\[\begin{gathered}\norm_{K(A)}(\dot a\otimes \dot b) - \disc_A(1\wedge a\wedge b,1\wedge a\wedge b) = \norm_{\discalg(A)}\bigl(s_1(ab)-\gamma(1,b,a)\bigr)\\
= s_1(ab)^2 - s_1(ab)\trace_{\discalg(A)}\bigl(\gamma(1,b,a)\bigr)+\norm_{\discalg(A)}\bigl(\gamma(1,b,a)\bigr).
\end{gathered}\]
We compute each of these terms in turn.
Now the norm of $\dot a \otimes \dot b$ in $K(A)$ is
\begin{align*}
\norm_{K(A)}(\dot a\otimes \dot b) &= (\dot a \otimes \dot b)(\dot b \otimes \dot a)= \dot a\otimes\dot b\otimes\dot b \otimes \dot a= q(\dot b)(\dot a\otimes\dot a)= q(\dot a)q(\dot b)\\
&= (s_1(a^2)-s_2(a))(s_1(b^2)-s_2(b))\\
&= s_1(a^2)s_1(b^2) - s_2(a)s_1(b^2) - s_1(a^2)s_2(b) + s_2(a)s_2(b).
\end{align*}
Next, we expand $\disc_A(1\wedge a\wedge b,1\wedge a\wedge b)$ as
\begin{align*}
 \det\begin{pmatrix}
  3 & s_1(a) & s_1(b) \\
  s_1(a) & s_1(a^2) & s_1(ab)\\
  s_1(b) & s_1(ab) & s_1(b^2)
 \end{pmatrix}
  &= \begin{array}{l}3s_1(a^2)s_1(b^2) - 3s_1(ab)^2\\
  \ -\ s_1(a)^2s_1(b^2) - s_1(a^2)s_1(b)^2 \\
  \ +\ 2s_1(a)s_1(ab)s_1(b).
  \end{array}
\end{align*}
We can use \cref{relations-for-sk} to expand $s_1(a)s_1(ab)s_1(b)$ and terms of the form $s_1(\,\cdot\,)^2$, and then simplify:
 \begin{align*} &\!\begin{array}{l}=\ 3s_1(a^2)s_1(b^2) - 3\bigl(s_1(a^2b^2)+2s_2(ab)\bigr)\\
  \quad -\ \bigl(s_1(a^2)+2s_2(a)\bigr)s_1(b^2) - s_1(a^2)\bigl(s_1(b^2)+2s_2(b)\bigr) \\
  \quad +\ 2\bigl(s_1(a^2b) s_1(b) +s_1(a) s_1(ab^2)   - s_1(a^2b^2)+  s_2(ab) + s_2(a)s_2(b)\bigr),
  \end{array}\\
  &\!\begin{array}{l}=\ s_1(a^2)s_1(b^2) - 5s_1(a^2b^2)-4s_2(ab) - 2s_2(a)s_1(b^2)-2s_1(a^2)s_2(b) \\
  \quad +\ 2s_1(a^2b) s_1(b) +2s_1(a) s_1(ab^2) + 2s_2(a)s_2(b).
  \end{array}
\end{align*}
Hence the difference $\norm_{K(A)}(\dot a \otimes \dot b) - \disc_A(1\wedge a\wedge b, 1\wedge a\wedge b)$ is
\[\begin{gathered}5s_1(a^2b^2)+ 4s_2(ab) + s_2(a)s_1(b^2)+s_1(a^2)s_2(b)\\ - 2s_1(a^2b)s_1(b) - 2s_1(a)s_1(ab^2) - s_2(a)s_2(b).\end{gathered}\]

Next we expand $s_1(ab)^2-s_1(ab)\trace_{\discalg(A)}\bigl(\gamma(1,b,a)\bigr) + \norm_{\discalg(A)}\bigl(\gamma(1,b,a)\bigr)$.
We computed the trace of $\gamma(1,b,a)$ above as $\gamma(1,b,a)+\gamma(1,a,b) = s_1(a)s_1(b) - s_1(ab)$.
Therefore the quantity $s_1(ab)^2 - s_1(ab)\trace_{\discalg(A)}\bigl(\gamma(1,b,a)\bigr)$ is
\[\begin{aligned}
 & s_1(ab)^2 - s_1(ab)\bigl(s_1(a)s_1(b) - s_1(ab)\bigr)
 = 2s_1(ab)^2 - s_1(a)s_1(ab)s_1(b) \\
 &\quad= 2\bigl(s_1(a^2b^2)+2s_2(ab)\bigr) - \bigl(s_1(a^2b) s_1(b) +s_1(a) s_1(ab^2) \\
 &\hspace{2in}  - s_1(a^2b^2)+  s_2(ab) + s_2(a)s_2(b)\bigr)\\
 &\quad= 3s_1(a^2b^2) + 3s_2(ab) - s_2(a)s_2(b) - s_1(a^2b)s_1(b) - s_1(a)s_1(ab^2).
\end{aligned}\]

Finally, the norm of $\gamma(1,b,a)$ may similarly be computed as
\begin{align*}
\norm(\gamma(1,b,a)) &= \gamma(1,b,a)\gamma(1,a,b)\\
&= \gamma(1,ab,ab) + \gamma(b,b,a^2) + \gamma(a,b^2,a)\\
&= s_2(ab)+s_1(a^2)s_2(b) - s_1(a^2 b)s_1(b) + s_1(a^2b^2)\\
&\quad + s_2(a)s_1(b^2) - s_1(a)s_1(a b^2) + s_1(a^2b^2)\\
&= 2s_1(a^2b^2)+s_2(ab) - s_1(a^2b)s_1(b) - s_1(a)s_1(ab^2)\\
&\quad + s_1(a^2)s_2(b) + s_2(a)s_1(b^2).
\end{align*}
So all together we have
\begin{align*}
\norm_{\discalg(A)}\bigl(s_1(ab) - \gamma(1,b,a)\bigr) &= 5s_1(a^2b^2) + 4s_2(ab) - s_2(a)s_2(b)\\
&\quad - 2s_1(a^2b)s_1(b) - 2s_1(a)s_1(ab^2)\\
&\quad + s_1(a^2)s_2(b) + s_2(a)s_1(b^2).
\end{align*}

So the isomorphism $D(A) \to \discalg(A)$ preserves the norm and trace of elements of the form $\dot a\otimes \dot b$, and it is then an isomorphisms of $R$-algebras, as we wanted to show.
\end{proof}

\subsection{The existence of $\Sigma_F$}

Next we prove \cref{lemma-Sigma-existence}, that the element $\det\bigl(\conjugate{x_i}{j}\bigr)_{i,j} + \Pf(F-\transpose{F})$ of $\Xbase^{\otimes n}$ is a multiple of $2$ by an element of $\fixpower{\Xbase}{n}{\A{n}}$.

\begin{proof}[Proof of \cref{lemma-Sigma-existence}]
 A permutation of the $x_i$ permutes the rows of the matrix $\bigl(\conjugate{ x_i}{j}\bigr)_{i,j=1}^n$, so its determinant is $\altgroup{n}$-invariant.  
 And since the entries of $F-\transpose{F}$ are $\permgroup{n}$-invariant, so is its Pfaffian $\Pf(F-\transpose{F})$.  
 Thus we need only show that the sum $\det\bigl(\conjugate{ x_i}{j}\bigr)_{i,j=1}^n + \Pf(F-\transpose{F})$ is a multiple of $2$ in $\Xbase^{\otimes n}$, i.e.\ that it is sent to $0$ in $\Xbase^{\otimes n}/(2)\cong (\Xbase/(2))^{\otimes n}$.  
 Now consider $\Pf(F-\transpose{F})$ modulo $2$.  
 Since $F_{ij}-F_{ji} = 2f_{ij}-B_{ij}\equiv B_{ij}$ for $i<j$, the Pfaffian $\Pf(F-\transpose{F})$ is congruent modulo $2$ to the Pfaffian of the alternating $n\times n$ matrix $B_0$ whose $ij$th entry is $B_{ij}$ if $i<j$ and $-B_{ij}$ if $i>j$.  
 Working in $\Xbase^{\otimes n}/(2)$, we can expand this Pfaffian as 
 \[
  \Pf(B_0) = \sum_{S\in\binom{n}{2,\dots,2}} \prod_{(i,j)\in S} B_{ij} = \sum_{S\in\binom{n}{2,\dots,2}} \prod_{(i,j)\in S} \sum_{k\neq\ell}\conjugate{ x_i}{k}\conjugate{ x_j}{\ell},
 \]
 where the outer sum is over all partitions $S$ of $\set{n}$ into $n/2$ pairs $\{i,j\}$ with $i\neq j$.  
 Expanding this sum, we obtain a sum of products of the form $\conjugate{ x_1}{k_1}\dots\conjugate{ x_n}{k_n}$; we claim that such a monomial appears an odd number of times in the sum if and only if the $k_i$ are all distinct.  
 
 To wit, suppose that $k_i\neq k_j$ for all $i\neq j$; 
 then for each partition $S$ of $\set{n}$ into ordered pairs, the monomial $\conjugate{ x_1}{k_1}\dots\conjugate{ x_n}{k_n}$ appears exactly once in $\prod_{(i,j)\in S} \sum_{k\neq\ell}\conjugate{ x_i}{k}\conjugate{ x_j}{\ell}$, namely by setting $k=k_i$ and $\ell=k_j$ for each $\{i,j\}\in S$.  
 There are $(n-1)!! = (n-1)(n-3)\dots5\cdot3\cdot1$ such partitions $S$, so the monomial $\conjugate{ x_1}{k_1}\dots\conjugate{ x_n}{k_n}$ appears an odd number of times in the expansion of $\Pf(B_0)$.
 
 On the other hand, suppose that $k_{i_1}=k_{i_2}$ for some $i_1\neq i_2$ in $\set{n}$.  
 We will produce a fixed-point-free involution $\tau$ on the subset of partitions $S$ for which $\conjugate{ x_1}{k_1}\dots\conjugate{ x_n}{k_n}$ appears in the expansion of $\prod_{\{i,j\}\in S} \sum_{k\neq\ell}\conjugate{ x_i}{k}\conjugate{ x_j}{\ell}$.  
 If $S$ is such a partition, then $k_i\neq k_j$ for each $\{i,j\}\in S$.  
 In particular, $\{i_1,i_2\}$ is not in $S$; 
 let $j_1$ and $j_2$ be such that $\{i_1,j_1\}$ and $\{i_2,j_2\}$ are elements of $S$. 
 Then set $\tau(S)$ to be the partition of $\set{n}$ differing from $S$ only in that it contains $\{i_1,j_2\}$ and $\{i_2,j_1\}$ instead of $\{i_1,j_1\}$ and $\{i_2,j_2\}$.  
 The monomial $\conjugate{ x_1}{k_1}\dots\conjugate{ x_n}{k_n}$ again appears in $\tau(S)$ since $k_{i_1}=k_{i_2}\neq k_{j_2}$ and $k_{i_2}=k_{i_1}\neq k_{j_1}$; 
 furthermore $\tau(S)$ never equals $S$ and $\tau(\tau(S))$ always equals $S$.  
 Thus there is an even number of such partitions $S$, and the monomial $\conjugate{ x_1}{k_1}\dots\conjugate{ x_n}{k_n}$ appears an even number of times in the expansion of $\Pf(B_0)$.
 
 We have shown that in $\Xbase^{\otimes n}/(2)$,
 \[\Pf(F-\transpose{F}) = \Pf(B_0) = \sum_{\sigma\in\permgroup{n}}\conjugate{x_1}{\sigma(1)}\dots\conjugate{x_n}{\sigma(n)} = \det\bigl(\conjugate{x_i}{j}\bigr)_{i,j=1}^n,\]
 since all other terms in the sum appear an even number of times, and hence vanish. Thus the sum $\det\bigl(\conjugate{x_i}{j}\bigr)_{i,j=1}^n + \Pf(F-\transpose{F})$ vanishes modulo $2$, so the sum is a multiple of $2$ in $\Xbase^{\otimes n}$.
 \end{proof}
 
\subsection{A determinant identity}

Lastly, we prove the claim in the proof of \ref{loos-iso} that $\det(B_{ij})_{i,j=1}^n = (1-n)\det\bigl(e_1(x_ix_j)\bigr)_{i,j=1}^n$ for $n$ even.

\begin{lemma}\label{lemma-disc-identity}
 Let $n$ be a natural number, and consider the following two matrices over the polynomial ring $\Z[x_1,\dots,x_n]^{\otimes n}$: 
 Let $B$ be the $n\times n$ matrix whose $ij$th entry is $B_{ij} = e_1(x_i)e_1(x_j)-e_1(x_ix_j)$, and let $D$ be the $n\times n$ matrix whose $ij$th entry is $D_{ij} = e_1(x_ix_j)$.  
 Then
 \[
 \det(B) = (-1)^n(1-n)\det(D).
 \]
\end{lemma}

\begin{proof}
 First, note that if $n=0$ the identity becomes the tautology $1=1$, so assume $n\geq 1$. 
 Let $v$ be the column vector whose $i$th entry is $e_1(x_i)$; then $B = v\transpose{v} - D$.  
 Thus we need only prove that $\det(D - v\transpose{v}) = (1-n)\det(D)$; 
 we will prove this in the larger ring $\Q[x_1,\dots,x_n]^{\otimes n}$.  
 Note that since $v\transpose{v}$ is a rank-$1$ matrix, the polynomial
 \[
 p(\lambda) = \det(D-\lambda v\transpose{v})
 \]
 has constant term $\det(D)$ and all terms of degree at least $2$ vanishing; this can be checked by changing to a basis in which $v\transpose{v}$ is in row-echelon form, or more elementarily by showing that all contributions to the determinant involving at least two factors of $\lambda$ must cancel. 
 Writing $p(\lambda)=\det(D) - \lambda a$, we would like to show that $a=n\,\det(D)$ so that $p(1) = (1-n)\det(D)$.
 To prove this, we will show that $p(1/n) = 0$, i.e.\ that $D-\frac{1}{n}v\transpose{v}$ is singular.
 
 Let $A$ be the $n\times\binom{n}{2}$ matrix with rows indexed by elements of $\set{n}$ and columns indexed by pairs $(k,\ell)$ with $1\leq k < \ell \leq n$, and whose $(i,(k,\ell))$th entry is $A_{i,(k,\ell)}=\conjugate{x_i}{k} - \conjugate{x_i}{\ell}$.
 Then we claim that $D-\frac{1}{n}v\transpose{v} = \frac1n A\transpose{A}$.
 
 Indeed, we have
 \begin{align*}
  (A\transpose A)_{ij} &= \sum_{k<\ell} \left(\conjugate{x_i}{k}-\conjugate{x_i}{\ell}\right)\left(\conjugate{x_j}{k}-\conjugate{x_j}{\ell}\right)\\
  &= \sum_{k<\ell} \left(\conjugate{x_i}{k}\conjugate{x_j}{k}+\conjugate{x_i}{\ell}\conjugate{x_j}{\ell}\right) - \sum_{k<\ell}\left(\conjugate{x_i}{k}\conjugate{x_j}{\ell}+\conjugate{x_i}{\ell}\conjugate{x_j}{k}\right)\\
  &= \sum_{k\neq\ell} \conjugate{x_i}{k}\conjugate{x_j}{k} - \sum_{k\neq\ell} \conjugate{x_i}{k}\conjugate{x_j}{\ell}\\
  &= (n-1)\sum_{k} \conjugate{x_i}{k}\conjugate{x_j}{k} - \sum_{k\neq\ell} \conjugate{x_i}{k}\conjugate{x_j}{\ell}\\
  &= n\sum_{k} \conjugate{x_i}{k}\conjugate{x_j}{k} - \sum_{k,\ell} \conjugate{x_i}{k}\conjugate{x_j}{\ell}\\
  &= n\sum_{k}\conjugate{x_i}{k}\conjugate{x_j}{k} - \sum_{k}\conjugate{x_i}{k}\sum_{\ell} \conjugate{x_j}{\ell}\\
  &= n e_1(x_i x_j) - e_1(x_i)e_1(x_j)\\
  &= n D_{ij} - v_iv_j.
 \end{align*}
 Thus it remains only to show that $A$ has rank less than $n$; 
 then $D-\frac{1}{n}v\transpose{v}$ will also have rank less than $n$. 
  So choose any $n$ distinct columns of $A$, and we will show that they are linearly dependent.  
  Indeed, such a choice amounts to choosing $n$ distinct pairs of distinct elements from $\set{n}$, which is the data of a graph with $n$ edges and $n$ vertices.  
  Such a graph must contain a cycle, since a tree with $n-1$ edges already spans $n$ vertices. 
   Then our collection of $n$ columns of $A$ contains a subset of the form (up to changing the overall sign of some columns)
 \[\left(\begin{matrix}
  \conjugate{x_1}{k_1} - \conjugate{x_1}{k_2}\\
  \vdots\\
  \conjugate{x_n}{k_1} - \conjugate{x_n}{k_2}
 \end{matrix}\right), \left(\begin{matrix}
  \conjugate{x_1}{k_2} - \conjugate{x_1}{k_3}\\
  \vdots\\
  \conjugate{x_n}{k_2} - \conjugate{x_n}{k_3}
 \end{matrix}\right), \dots,
 \left(\begin{matrix}
  \conjugate{x_1}{k_{\ell}} - \conjugate{x_1}{k_1}\\
  \vdots\\
  \conjugate{x_n}{k_{\ell}} - \conjugate{x_n}{k_1}
 \end{matrix}\right),
\]
 and the sum of these is $0$.  
 Thus every $n\times n$ minor of $A$ vanishes, so $A$ has rank at most $n-1$, and so does $D-\frac1n v\transpose{v}$.  
 Then $\det(D-\frac1n v\transpose{v})=0$, so $\det(D-v\transpose{v}) = (1-n)\det(D)$ and $\det(B) = (-1)^n(1-n)\det(D)$.
\end{proof}

\bibliographystyle{acm}
\bibliography{RefList}

\end{document}